\documentclass[12pt]{amsart}

\usepackage{float}
\usepackage[english]{babel}
\usepackage{amssymb, latexsym, amsmath, amsthm}
\usepackage[left=.5cm, right=.5cm, top=1cm]{geometry}
\usepackage{float}
\usepackage{tikz-cd}
\usepackage{caption}
\usepackage{float}
\usepackage{subcaption}
\usepackage{enumitem}
\usepackage{verbatim}
\usepackage{blindtext}
\usepackage{bbm}
\usepackage{longtable}
\usepackage{fullpage} 
\usepackage{hyperref}
\usepackage{enumitem}
\usepackage{comment}
\usepackage{mathtools}
\usepackage{url}
\usepackage{geometry}
\usepackage{ragged2e}
\usepackage{xcolor}

\hypersetup{backref=true,       
    pagebackref=true,               
    hyperindex=true,                
    colorlinks=true,                
    breaklinks=true,                
    urlcolor= black,                
    linkcolor= black,                
    bookmarks=true,                 
    bookmarksopen=false,
    filecolor=black,
    citecolor=black,
    linkbordercolor=red}
\AtBeginDocument{\hypersetup{pdfborder={0 0 1}}}

\newcommand{\N}{\mathbbm{N}}

\newcommand{\compconj}[1]{\overline{#1}}

\newtheorem{thm}{Theorem}[section]
\newtheorem{prop}[thm]{Proposition}
\newtheorem{lemma}[thm]{Lemma}

\newtheorem{cor}[thm]{Corollary}

\newtheorem{defn}[thm]{Definition}
\newtheorem{Rem}[thm]{Remark}

\newtheorem{?}[thm]{Problem}

\theoremstyle{plain}
\theoremstyle{definition}

\begin{document}

\title{An analogue of Weil's converse theorem for harmonic Maass forms of polynomial growth}

\author{Karam Deo Shankhadhar, Ranveer Kumar Singh}

\address[Karam Deo Shankhadhar]{Department of Mathematics, Indian Institute of Science Education and Research Bhopal,
Bhopal Bypass Road,  Bhauri,
Bhopal 462 066,
Madhya Pradesh, India}
\email{karamdeo@iiserb.ac.in, karamdeo@gmail.com}

\address[Ranveer Kumar Singh]{{NHETC, Department of Physics and Astronomy, Rutgers University,
126 Frelinghuysen Rd., Piscataway NJ08855, USA}}
\email{ranveersfl@gmail.com}




\maketitle

\begin{abstract}
\noindent We construct a family of harmonic Maass forms of polynomial growth of any level 
corresponding to any cusp whose shadows are Eisenstein series of integral weight. 
We further consider Dirichlet series attached to a 
harmonic Maass form of polynomial growth, study its analytic properties, and prove an 
analogue of Weil's converse theorem.

\end{abstract}

\section{Introduction}
Harmonic Maass forms and mock modular forms were first glimpsed in the enigmatic ``deathbed" 
letter that Ramanujan wrote to Hardy in 1920. For many decades, very little was understood about the functions written in this letter and no comprehensive theory was available to explain them. Finally, Zwegers \cite{SZ} in his thesis showed that Ramanujan's mock theta functions can be realized as holomorphic parts of some special families of nonholomorphic modular forms and fit very well with several important theories in mathematics. At almost the same time, Bruinier and Funke \cite{BF} wrote a very important paper in which they defined the notion of a harmonic Maass form and the nonholomorphic modular 
forms constructed by Zwegers turned out to be weight $1/2$ harmonic Maass forms. Harmonic Maass forms 
of manageable growth generalize harmonic Maass forms by weakening the cusp conditions. 
For a systematic development and detailed treatment of general theory of 
harmonic Maass forms we refer the reader to the book of Bringmann, Folsom, Ono and Rolen \cite{Ono1}, 
which is essentially self-contained.

In this article, we study those harmonic Maass forms of manageable growth which have at most polynomial growth at all the cusps. We denote the space of harmonic Maass forms of polynomial growth of weight $k$ for the group $\Gamma_0(N)$ with character $\chi$ by $H_k^\#(N, \chi)$. 
We construct a family of examples in $H_k^\#(N, \chi)$ whose shadows are Eisenstein series. Next, we consider Dirichlet series attached to the functions in the space $H_k^\#(N, \chi)$, study analytic properties, and prove an analogue of Weil's converse theorem for this space.   

A converse theorem in the theory of automorphic forms refers to the 
equivalence of Dirichlet series satisfying certain analytic properties, on the one hand,
and automorphic forms over some group, on the other. Most familiar is the converse theorem due to Hecke 
\cite{Hecke}, which establishes an equivalence between modular forms on $\mathrm{SL}_2(\mathbb{Z})$ and Dirichlet series 
satisfying a certain functional equation as well as meromorphic continuation and a certain boundedness 
property in vertical strips. In our context, the meaning of a converse theorem is best illustrated  by Weil's converse theorem for modular forms over congruence subgroups $\Gamma_0(N)$ \cite{Weil}, which is a very significant generalization of the corresponding Hecke's theorem for $N = 1$.  Other results of this kind are 
Maass' converse theorem for 
Maass waveforms of level $1$ \cite{Maass}, its generalization to $\Gamma_0(N)$ by Neururer and Oliver
\cite{NO}, 
converse theorems for Jacobi forms \cite{KMS, martin96, martin-osses},  
and Maass Jacobi forms \cite{SH}. The converse theorem 
for $\mathrm{GL}_n$ is a great achievement of several authors through a string of papers \cite{CP, JL, JPS}. 

\subsection{Statement of the main result}\label{statement}

In order to state our converse theorem for the space $H_k^\#(N, \chi)$, we first introduce some basic 
notations that will be used throughout the article. Let $\mathbb{C}$ denote the complex plane. For each 
$z \in \mathbb{C}$, denote the real and imaginary parts of $z$ by $\text{Re}(z)$ and $\text{Im}(z)$, respectively. Let $i=\sqrt{-1}$ and $\mathbb{H}= \{\tau \in \mathbb{C}:\; \text{Im}(\tau)> 0\}$ be the upper half-plane. For $\tau \in \mathbb{H}$, let $\tau = u+i v$ and $q=e^{2\pi i \tau}$. 

Let us fix two integers $k$ and $N$ with $N \geq 1$ and a Dirichlet character $\chi$ modulo $N$ 
such that $\chi(-1)=(-1)^k$. 
We assume that $k$ 
is a {\it negative integer} (see Proposition \ref{prop:hashproperty} and Remark \ref{Rem:signk}). 
Let $f$ and $g$ be two functions on 
$\mathbb{H}$ given by the following formal Fourier series
\begin{equation}\label{eq:fgfourier}
\begin{split}
f(\tau)=\sum_{n=0}^{\infty}c_f^+(n)q^n+c_f^-(0)v^{1-k}+\sum_{n<0}c^-_f(n)\Gamma(1-k,-4\pi nv)q^n,\\
g(\tau)=\sum_{n=0}^{\infty}c_g^+(n)q^n+c_g^-(0)v^{1-k}+\sum_{n<0}c^-_g(n)\Gamma(1-k,-4\pi nv)q^n,
\end{split}
\end{equation}
with $c_f^{\pm}(n), c_g^{\pm}(n)$ bounded by $O\left(|n|^{\alpha}\right), n \in \mathbb{Z}$, for some 
$\alpha \geq 0$.

For any $\nu \in\mathbb{R}$, define 
\begin{equation}\label{eq:wnu}
W_{\nu}(s)=\int_{0}^{\infty}\Gamma(\nu,2x)e^{x} x^s \frac{dx}{x}, \ \ {\rm Re} (s) > 0,
\end{equation}
where $\Gamma(\nu, 2x)$ is the incomplete gamma function given by \eqref{2.3}. 
Now, we associate completed Dirichlet series $\Lambda_N(f,s)$ and 
$\Omega_N(f,s)$ to the function $f$ as follows.
\begin{gather*}
\Lambda_N(f,s) = \left(\frac{\sqrt{N}}{2 \pi}\right)^{s} \left[\Gamma(s) L^+(f,s)+W_{1-k}(s)L^-(f,s)\right],\\
\Xi_N(f,s) = \left(\frac{\sqrt{N}}{2 \pi}\right)^{s} \left[\Gamma(s+1) L^+(f,s)-W_{1-k}(s+1)L^-(f,s)\right],\\
\Omega_N(f,s) = -2~\Xi_N(f,s)+k\Lambda_N(f,s),
\end{gather*}
where
\begin{equation*}
L^{+}(f,s)=\sum_{n=1}^{\infty}\frac{c^{+}_f(n)}{n^s},~~~
L^{-}(f,s)=\sum_{n=1}^{\infty}\frac{c^-_f(-n)}{n^s},
\end{equation*}  
are Dirichlet series attached to $f$. Similarly, we associate Dirichlet series $L^{\pm}(g,s), \Lambda_N(g,s)$ and 
$\Omega_N(g,s)$ to the function $g$. 

As in the case of classical modular forms, we twist the Fourier series $f(\tau)$ by a Dirichlet character $\psi$ with conductor $m_\psi$ to get the following twisted Fourier series.
\begin{equation}\label{eq:ftwistedfourier}
f_\psi(\tau) = \sum_{n=0}^{\infty}\psi(n)c_f^+(n)q^n+\psi(0)c_f^-(0)v^{1-k}+\sum_{n<0}\psi(-n)c^-_f(n)
\Gamma(1-k,-4\pi nv)q^n.
\end{equation}
Again, we consider the Dirichlet series 
$$
L^{\pm}(f,s,\psi) = \sum_{n=1}^\infty \frac{\psi(n)c_f^{\pm}(\pm n)}{n^s}
$$
and the completed Dirichlet series $\Lambda_N(f,s,\psi), \Omega_N(f, s, \psi)$ associated to $f_\psi$. Similarly,
we twist the Fourier series $g$ to $g_\psi$ and attach the Dirichlet series 
$L^{\pm}(g,s,\psi), \Lambda_N(g,s,\psi), \Omega_N(g, s, \psi)$ to $g_\psi$. Note that if $\psi$ is the trivial character with conductor $m_\psi=1$ then  
all the twisted Dirichlet series will be same as the Dirichlet series attached to $f$ and $g$.
 
Finally, we denote by $\mathcal{P}$ a set of odd prime numbers or $4$ which are relatively prime to $N$ and
whose intersection with every arithmetic progression of the form $a+nb$ of two coprime integers $a,b$, 
is non-empty. Also, let $\omega(N)$ be the standard 
Fricke involution defined as follows
$$
f|_k \omega(N) (\tau)= 
f|_k \begin{pmatrix} 0&-1\\N&0 \end{pmatrix} (\tau)= N^{k/2} (N\tau)^{-k} f\left(-\frac{1}{N\tau}\right).
$$

Now we state the main theorem of this paper which
characterize the Dirichlet series associated to the functions in the space $H_k^\#(N, \chi)$
and hence can be considered as an analogue of Weil's converse theorem in this case.               
 
\begin{thm}\label{thm:main}
Let $k, N, \chi$ be as above. Let $f$ and $g$ be two functions on $\mathbb{H}$ given by the formal 
Fourier series as in \eqref{eq:fgfourier}
with $c_f^{\pm}(n), c_g^{\pm}(n)$ bounded by $O\left(|n|^{\alpha}\right)$ for some $\alpha\geq 0$. Then the following two statements are equivalent.
\begin{enumerate}
\item The functions $f$ and $g$ belong to $H^{\#}_k(N,\chi)$ and 
$H^{\#}_k(N,\overline{\chi})$, respectively, and $f|_k\omega(N)=g$.
\item 
\begin{enumerate}
\item Each one of the completed Dirichlet series $\Lambda_N(f,s), \Lambda_N(g,s), \Omega_N(f,s),\Omega_N(g,s)$ admits a meromorphic continuation to the whole complex plane and satisfies the 
functional equation 
\begin{equation*}
\Lambda_N(f,s)=i^{k} \Lambda_N(g,k-s), ~~ \Omega_N(f,s)=-i^{k} \Omega_N(g,k-s).
\end{equation*}
Moreover, each of the following functions is entire and bounded in vertical strips
\begin{gather*}
\Lambda^*_N(f,s) =\Lambda_N(f,s)+\frac{c_f^+(0)}{s}+\frac{c_g^+(0) i^{k}}{k-s}+\frac{c^-_f(0)}{N^{\frac{1-k}{2}}}\frac{1}{s-k+1}+\frac{c_g^-(0)i^k}{N^{\frac{1-k}{2}}}\frac{1}{1-s},\\
\Lambda^*_N(g,s) =\Lambda_N(g,s)+\frac{c_g^+(0)}{s}+\frac{c_f^+(0) i^{-k}}{k-s}+\frac{c^-_g(0)}{N^{\frac{1-k}{2}}}\frac{1}{s-k+1}+\frac{c_f^-(0)i^{-k}}{N^{\frac{1-k}{2}}}\frac{1}{1-s},\\
\Omega^*_N(f,s) =\Omega_N(f,s)+k\left(\frac{c_f^+(0)}{s}-\frac{c_g^+(0) i^{k}}{k-s}+\frac{c^-_f(0)}{N^{\frac{1-k}{2}}}\frac{1}{s-k+1}-\frac{c_g^-(0)i^k}{N^{\frac{1-k}{2}}}\frac{1}{1-s}\right),\\
\Omega^*_N(g,s) =\Omega_N(g,s)+k\left(\frac{c_g^+(0)}{s}-\frac{c_f^+(0) i^{-k}}{k-s}+\frac{c^-_g(0)}{N^{\frac{1-k}{2}}}\frac{1}{s-k+1}-\frac{c_f^-(0)i^{-k}}{N^{\frac{1-k}{2}}}\frac{1}{1-s}\right).
\end{gather*}
\item 
For any primitive Dirichlet character $\psi$ with conductor $m_{\psi}\in\mathcal{P}$, each one of the completed Dirichlet series $\Lambda_{N}(f,s,\psi),\Lambda_{N}(g,s,\psi), \Omega_{N}(f,s,\psi), \Omega_{N}(g,s,\psi)$  
can be analytically continued to the whole $s$-plane, bounded on any vertical strip, and satisfies the 
functional equation
\begin{equation*}
\begin{split}
\Lambda_{N}(f,s,\psi)=i^{k} C_{\psi} \Lambda_{N}(g,k-s ,\bar{\psi}),\\
\Omega_{N}(f,s,\psi)=-i^{k} C_{\psi} \Omega_{N}(g,k-s ,\bar{\psi}),
\end{split}
\end{equation*}
where
\begin{equation*}
C_{\psi}=C_{N, \psi} =\chi(m) \psi(-N) \tau(\psi) / \tau(\overline{\psi}) =\chi(m) \psi(N) \tau(\psi)^{2} / m.
\end{equation*}
Here, $\tau(\psi) = \displaystyle\sum_{a=0}^{m_\psi -1} \psi(a) e^{2 \pi i a/m_\psi}$
is the Gauss sum of $\psi$. Similarly, $\tau(\overline{\psi})$ denotes the Gauss sum of $\overline{\psi}$.
\end{enumerate}
\end{enumerate}
\end{thm}

\begin{Rem}
\begin{enumerate}
\item While proving the direct part $(1) \implies (2)$ in \S 5.1, we establish (2)(b) for any primitive Dirichlet character whose conductor is coprime to the level. 
So for any $f, g$ satisfying the condition (1), the analytic properties given in (2)(b) are true for 
Dirichlet series twisted by any primitive Dirichlet character with conductor coprime to $N$.  
\item In the proof of the converse part $(2) \implies (1)$ presented in \S 5.2, we do not need condition (2)(b) to 
be satisfied for all primitive Dirichlet characters with conductor $m_\psi \in \mathcal{P}$. In fact for a fixed 
$N\geq 1$, having the condition (2)(b) only for finitely many primitive Dirichlet characters is sufficient. 
The same is true in the case of modular forms as well, evident from the proof of \cite[Theorem 4.3.15]{miyake}. 
For any fixed $N \geq 1$, by using a computational platform one can determine finitely many $m_\psi$ which are enough to establish the converse part. For example, for $N=7, 11$ we need to assume 2(b) for 
$m_\psi=11,17,19, 23, 29, 41$ and $m_\psi = 13,17, 19, 23, 29, 31, 37, 47, 59, 71$ respectively. 
Unfortunately, we are not able to find such a finite set which works uniformly for all $N$.  
\end{enumerate}
\end{Rem}

This article is organized as follows. In the next section, we recall basic facts about Harmonic Maass forms of manageable growth, study differential operators and Fricke involution for these forms, and prove some 
lemmas which are used throughout the paper. In \S 3, we describe the space $H_k^\#(N,\chi)$ and 
construct interesting examples in this space. In \S 4, we discuss some intermediate results which will be useful 
in establishing Theorem \ref{thm:main}. In particular, we establish an analogue of Hecke's converse theorem 
for the space $H_k^\#(N,\chi)$ which might be of independent interest as well in some other context.
In \S 5, we prove the direct and converse parts of 
Theorem \ref{thm:main}. To prove Theorem \ref{thm:main}, 
we essentially follow the proof of Weil's converse theorem presented in
\cite[Theorem 4.3.15]{miyake}. The key differences are the holomorphicity of 
functions and the shape of Fourier series expansions. Certain differential operators
and some crucial properties of incomplete Gamma functions are helpful to overcome 
the technical difficulties presented due to the nature of harmonic Maass forms.    
  
\section{Notation and Preliminaries}

Throughout this section $k, N \in \mathbb{Z}$ with $N \geq 1$. Let $\mathrm{GL}^+_2(\mathbb{Q})$ be the group of $2\times 2$ matrices with rational entries and positive determinant.  
Let $\mathrm{SL}_2(\mathbb{Z})$ be the subgroup of $\mathrm{GL}_2^+(\mathbb{Q})$ of matrices having integer entries and determinant $1$.
For any positive integer $N$, let 
\begin{equation*}
\Gamma_0(N) = \left\{\begin{pmatrix}
a&b\\c&d\\
\end{pmatrix}\in \mathrm{SL}_2(\mathbb{Z}):c~\equiv ~0~(\text{mod} ~N)\right\}.
\end{equation*}

For any $\gamma=\left(\begin{smallmatrix}
a&b\\c&d
\end{smallmatrix}\right)
\in \mathrm{GL}^+_2(\mathbb{Q})$ and $k\in\mathbb{Z}$, we define the weight $k$ slash operator as follows.
If $f:\mathbb{H}\longrightarrow\mathbb{C}$ is a function, define 
\begin{equation*}
f|_k\gamma(\tau)=(\text{det}\gamma)^{k/2}(c\tau+d)^{-k}f(\gamma\tau)
\label{2.1}
\end{equation*}
where $\gamma\tau=(a\tau+b)(c\tau+d)^{-1}$. 

A holomorphic function $f:\mathbb{H}\longrightarrow\mathbb{C}$ is called a weakly holomorphic modular form (respectively modular form and cusp form) of weight $k$, level $N$ and character $\chi$ if $f|_k\gamma=\chi(d)f$ for every $\gamma\in\Gamma_0(N)$ and $f$ is meromorphic (respectively, holomorphic and vanishes) at all cusps of $\Gamma_0(N)$. We denote the $\mathbb{C}$-vector space of weakly holomorphic modular forms (respectively, modular and cusp forms) by
$M_k^!(N,\chi)$ (respectively, $M_k(N,\chi)$ and $S_k(N,\chi))$. 
 
Define the {\em weight-$k$ hyperbolic Laplacian operator} $\Delta_k$ on $\mathbb{H}$ by 
\begin{equation*}
\Delta_k = -v^2\left(\frac{\partial^2}{\partial u^2}+\frac{\partial^2}{\partial v^2}\right)+ikv\left(\frac{\partial}{\partial u}+i\frac{\partial}{\partial v}\right)=-4v^2\frac{\partial}{\partial \tau}\frac{\partial}{\partial \bar\tau}+2ikv\frac{\partial}{\partial \bar\tau}, \ \tau=u+i v.
\end{equation*}
Following Bruiner and Funke \cite{BF}, we define harmonic Maass forms.
\begin{defn}\label{defn:harmonic}
A smooth function $f:\mathbb{H}\rightarrow\mathbb{C}$ is called a {\em harmonic Maass form} of weight $k$, level $N$ and character $\chi$ if 
\begin{enumerate}[label=(\roman*)]
\item $f|_k\gamma=\chi(d)f$ for every $\gamma\in\Gamma_0(N)$.
\item $\Delta_k(f)=0$.
\item There exists a polynomial $P_{f}(\tau)\in\mathbb{C}[q^{-1}]$ such that $f(\tau)-P_{f}(\tau)=
O(e^{-\epsilon v})$ as $v\rightarrow\infty$ for some $\epsilon>0$. Analogous conditions are required at all other 
cusps of $\Gamma_0(N)$,equivalently, for any $\gamma \in \mathrm{SL}_2(\mathbb{Z})$ there exists a polynomial $P_{f,\gamma}(\tau)\in\mathbb{C}[q^{-1}]$ such that $f|_k\gamma(\tau)-P_{f,\gamma}(\tau)=
O(e^{-\epsilon v})$ as $v\rightarrow\infty$ for some $\epsilon>0$. 
\end{enumerate}
If the third condition in the above definition is replaced by $f(\tau)=O(e^{\epsilon v})$  
as $v\rightarrow\infty$ for some $\epsilon>0$ and analogously at all other cusps, then $f$ is said to be a 
{\em harmonic Maass form of manageable growth}. The space of harmonic Maass forms of weight $k$, level $N$ and character $\chi$ is denoted by $H_k(N,\chi)$ and that of harmonic Maass forms of manageable growth by $H_k^{!}(N,\chi)$. 
\end{defn}
Any $f\in H_k^{!}(N, \chi) (k \neq 1)$ admits the following Fourier series expansion at the cusp $i\infty$ 
\cite[Lemma 4.3]{Ono1}. 
\begin{equation}
f(\tau)=f(u+iv)=\sum\limits_{n>> -\infty}c_f^{+}(n)q^n+ c_f^{-}(0)v^{1-k}+\sum\limits_{\substack{n<< \infty\\n\neq 0}}c_f^{-}(n)\Gamma(1-k,-4\pi nv)q^n,
\label{2.2}
\end{equation}
where $\Gamma(1-k,-4\pi nv)$ is the incomplete gamma function defined by \eqref{2.3}. The notation $\sum_{n >> -\infty}$ means that $\sum_{n=\alpha_f}^\infty$ for some $\alpha_f \in \mathbb{Z}$. The other notation $\sum_{n << \infty}$ is analogous. In particular, if $f\in H_k(N, \chi)$ then the above expansion takes the following shape. 
\begin{equation*}
f(\tau)=f(u+iv)=\sum\limits_{n>> -\infty}c_f^{+}(n)q^n+\sum\limits_{\substack{n<0}}c_f^{-}(n)\Gamma(1-k,-4\pi nv)q^n.
\end{equation*}
The incomplete gamma function $\Gamma(s,z)$ \cite[Page 63]{Ono1} is given by 
\begin{equation}
\Gamma(s,z)\coloneqq\int\limits_z^{\infty}e^{-t}t^s\frac{dt}{t}
\label{2.3}
\end{equation}
for Re$(s) > 0$ and $z\in\mathbb{C}$ (or $s\in\mathbb{C}$ and $z\in\mathbb{H}$). 
For any $z \neq 0$, it can be analytically continued in $s$ via the functional equation 
$\Gamma(s+1, z)=s\Gamma(s,z)+z^s e^{-z}$.
The function $\Gamma(s,z)$ satisfies the following asymptotic behaviour \cite[Eq. 4.6]{Ono1}.  
\begin{equation}
\Gamma(s,x) \sim x^{s-1}e^{-x}\hspace{0.5cm} \textrm{as}\  x \in \mathbb{R} \ {\rm and}\ |x|\rightarrow \infty. 
\label{gamma}
\end{equation}
Suppose $\nu$ is a positive integer. From \cite[\S 8.5, Eq. 8.69]{AW}, we have
\begin{equation}\label{eq:gammasum}
\Gamma(\nu, x)=\Gamma(\nu) e^{-x} \sum_{l=0}^{\nu-1} \frac{x^l}{l!}.
\end{equation}

Let $\rho$ be any cusp of $\Gamma_0(N)$ and $\Gamma_\rho = \{g \in \Gamma_0(N) | g \rho = \rho\}$. 
Let $\gamma_\rho \in \mathrm{SL}_2(\mathbb{Z})$ such that $\gamma_\rho (i\infty)=\rho$. Then 
$\gamma_\rho^{-1} \Gamma_\rho \gamma_\rho$ fixes $i\infty$ and is hence generated by $-I_2$ and $\begin{pmatrix} 
1&t_\rho\\0&1\end{pmatrix}$ for some positive integer $t_\rho$. The integer $t_\rho$ is called the {\em width} of the cusp 
$\rho$. Let $g_\rho \in \Gamma_\rho$ such that
$\gamma_\rho^{-1} g_\rho \gamma_\rho = \begin{pmatrix}1&t_\rho\\0&1\end{pmatrix}$. For any $f\in H_k^!(N,\chi)$, 
we have
$$
(f|_k\gamma_\rho)|_k \begin{pmatrix}1&t_\rho\\0&1\end{pmatrix} = (f|_k g_\rho)|_k\gamma_\rho = 
\chi(d_\rho) f|_k\gamma_\rho, \ \ g_\rho = \begin{pmatrix} a_\rho & b_\rho\\c_\rho & d_\rho \end{pmatrix}.
$$
Let $\kappa_\rho \in [0,1)$ such that $\chi(d_\rho) = e^{2\pi i \kappa_\rho}$. The real number $\kappa_\rho$ is called the 
{\em cusp parameter} (cf. \cite[\S 3.7]{GH}).  

\begin{lemma}\label{fe}
Let $f\in H_k^!(N,\chi) (k \neq 1)$ and $\rho, \gamma_\rho, t_{\rho}, \kappa_\rho$ be as above. 
Then the Fourier series expansion of $f$ at the cusp $\rho$ will have the following shape.
\begin{equation}\label{eq:fgammarho}
(f|\gamma_{\rho})(\tau)=\sum_{n>>-\infty}c_f^+(n)q^{\frac{n+\kappa_{\rho}}{t_{\rho}}}+c_f^-(0)v^{1-k}q^{\frac{\kappa_{\rho}}{t_{\rho}}}+\sum_{{n<<\infty}\atop{n \neq 0}}c_f^-(n)\Gamma(1-k,-4\pi nv/t_{\rho})q^{\frac{n+\kappa_{\rho}}{t_{\rho}}}.
\end{equation} 
\end{lemma}
\begin{proof}
From the discussion before the lemma, we have  
\[
\left(e^{-2\pi i\kappa_{\rho}\tau/t_{\rho}}f|_k\gamma_{\rho}\right)|_k\left(\begin{smallmatrix}1&t_{\rho}\\0&1\end{smallmatrix}\right)=e^{-2\pi i\kappa_{\rho}\tau/t_{\rho}}f|_k\gamma_{\rho}.
\] 
Now by following the proof of \cite[Lemma 4.3]{Ono1} (for detailed calculation, see \cite[Proof of Theorem 4.6]{RKS}), we see that $e^{-2\pi i\kappa_{\rho}\tau/t_{\rho}}f|\gamma_{\rho}$ has Fourier expansion of the form 
\[
\sum_{n>>-\infty}c_f^+(n)q^{n/t_{\rho}}+c_f^-(0)v^{1-k}+\sum_{{n<<\infty}\atop{n \neq 0}}c_f^-(n)\Gamma(1-k,-4\pi nv/t_{\rho})q^{n/t_{\rho}}.
\] 
The lemma now follows.
\end{proof}

\begin{Rem}\label{lemma 4.1}
Let $f \in H^!_{k}(N,\chi)$ and $\alpha \in \mathrm{GL}_2^{+}(\mathbb{Q})$. Then $f|_k\alpha$ satisfies similar growth condition as 
$f$ at every cusp.
Let $s \in \mathbb{Q}\cup\{i\infty\}$ be a cusp and $s=\beta(i\infty)$ for some $\beta\in \mathrm{SL}_2(\mathbb{Z})$. 
We can assume that 
$\alpha\beta$ has integer entries by multiplying $\alpha\beta$ by an integer without affecting $f|_k\alpha\beta$. Then
there exists a matrix $\gamma\in \mathrm{SL}_2(\mathbb{Z})$ and $a,b,d \in \mathbb{Z}$ with $a\geq 1, d \geq 1$ 
$($cf. \cite[Problem 7.1.5]{Ram murty}$)$ such that 
\[
\alpha \beta=\gamma\begin{pmatrix}a & b \\0 & d
\end{pmatrix}.
\] 
Since $f$ has a Fourier series expansion of the form \eqref{eq:fgammarho} at every cusp, we have 
$\kappa_\gamma \in [0,1)$ and a positive integer $t_\gamma$, both depending on $\gamma$, such that
\[
\begin{split}
f|_k\alpha\beta (\tau) & =f|_k\gamma |_k\left(\begin{array}{ll}
a & b \\
0 & d
\end{array}\right) (u+ i v) =\\
&\left.\left(\sum\limits_{n>> -\infty}{c_f^{+}}(n)q^{\frac{n+\kappa_\gamma}{t_\gamma}}+ {c_f^{-}}(0)v^{1-k} q^{\frac{\kappa_\gamma}{t_\gamma}}+\sum\limits_{\substack{n<< \infty\\n\neq 0}}{c_f^{-}}(n)\Gamma(1-k,-4\pi nv/t_\gamma)q^{\frac{n+\kappa_\gamma}{t_\gamma}}\right)\right|_k\left(\begin{array}{ll}
a & b \\
0 & d
\end{array}\right)\\
&=e^{2 \pi i \frac{b \kappa_\gamma}{d t_\gamma}}\left({a}/{d}\right)^{k/2}\Bigg{[}\sum\limits_{n>> -\infty}\widetilde{c_f^{+}}(n)q^{\frac{a(n+\kappa_\gamma)}{d t_\gamma}}+ \left({a}/{d}\right)^{1-k}{c_f^{-}}(0)v^{1-k}q^{\frac{a\kappa_\gamma}{d t_\gamma}}\\&\hspace{6cm}+\sum\limits_{\substack{n<< \infty\\n\neq 0}}\widetilde{c_f^{-}}(n)\Gamma(1-k,-4\pi anv/d)q^{\frac{a(n+\kappa_\gamma)}{d t_\gamma}}\Bigg{]},
\end{split}
\]
where $\widetilde{c_f^\pm}(n) = e^{2 \pi i b n / d t_\gamma}c_f^\pm(n)$. Since $a>0$ and $d>0$, we have the claimed assertion.
\end{Rem}

Next, we discuss some useful differential operators acting on the space $H_k^!(N,\chi)$. Following \cite[Chapter 5]{Ono1}, we define the {\it Maass raising operator $R_k$} and the {\it Maass lowering operator $L_k$}.
\begin{equation*}\label{MaassRL}
\begin{split}
R_k=2i\frac{\partial}{\partial \tau} + \frac{k}{v} = i\left(\frac{\partial}{\partial u} - i\frac{\partial}{\partial v}\right)+\frac{k}{v}.\\
L_k= -2 i v^2 \frac{\partial}{\partial {\bar \tau}} = -i v^2 \left(\frac{\partial}{\partial u} + i\frac{\partial}{\partial v}\right).
\end{split}
\end{equation*}
From \cite[Lemma 5.2]{Ono1}, we have the following identity. 
\begin{equation}\label{RLdelta}
-\Delta_k = L_{k+2} R_k + k = R_{k-2} L_k.
\end{equation}
 We define the {\em shadow operator} $\xi_k$ (cf. \cite[\S 5.2]{Ono1}) as follows: 
\begin{equation*}
\xi_k = 2iv^k\compconj{\frac{\partial}{\partial\bar\tau}}.
\end{equation*} 
It is related to the lowering operator $L_k$ \cite[Page 74]{Ono1} and the Laplacian operator $\Delta_k$ \cite[Eq. 2, Page 1]{HP} as follows. 
\begin{equation}
\xi_k=v^{k-2} \overline{L}_k, \ \ \Delta_k = -\xi_{2-k}\circ\xi_{k}.
\label{2.5}
\end{equation}
Suppose $k \neq 1$. Then $\xi_{k}: H_{k}^{!}(N, \chi) \rightarrow M_{2-k}^!(N, \overline{\chi})$ is a surjective linear map \cite[Theorem 5.10]{Ono1}, \cite[Lemma 2.2]{BO}. Moreover, for any $f \in H_{k}^{!}(N, \chi)$ with Fourier series expansion as in \eqref{2.2}, we have 
\begin{equation}
\xi_{k}(f(\tau))=(1-k)\compconj{c_f^-(0)}-(4\pi)^{1-k}\sum\limits_{n>>-\infty}\compconj{c_f^-(-n)}n^{1-k}q^n.
\label{2.7}
\end{equation}
Let $D = \frac{1}{2\pi i} \frac{\partial}{\partial\tau}$. For $k \leq 0$, we have the differential operator $D^{1-k}:H^!_k(N,\chi) \rightarrow M^!_{2-k}(N,\chi)$, called the {\em Bol operator} \cite[Theorem 5.5]{Ono1}. 
We have \cite[Lemma 5.3]{Ono1}
\begin{equation}\label{bolr}
D^{1-k} = \frac{1}{(-4\pi)^{k-1}}R_{k}^{1-k}.
\end{equation}
From \cite[Theorem 5.5]{Ono1}, for any $f \in H_{k}^{!}(N, \chi)$ with Fourier series expansion as in \eqref{2.2}, we have 
\begin{equation}\label{eq:DFourier}
D^{1-k}(f(\tau))=-(4 \pi)^{k-1}(1-k) ! c_{f}^{-}(0)+\sum_{n \gg-\infty} c_{f}^{+}(n) n^{1-k} q^{n}.
\end{equation}

\begin{lemma}\label{lemma:commutation}
Let $f$ be a smooth function on $\mathbb{H}$ and $\alpha \in$ $GL_2^{+}(\mathbb{Q})$. Then we have 
\[
\begin{split}
R_k(f|_k\alpha) = R_k(f)|_{k+2} \alpha, \ L_k(f|_k\alpha)=L_k(f)|_{k-2}\alpha.\\
\Delta_k (f|_k\alpha) = \Delta_k(f)|_k\alpha, \ \xi_k(f|_k\alpha) = \xi_k(f)|_{2-k}\alpha.
\end{split}
\]
Moreover, if $k \leq 0$ then $D^{1-k}(f|_k \alpha) = D^{1-k}(f)|_{2-k} \alpha$.
\end{lemma}
\begin{proof}
In view of \eqref{RLdelta}, \eqref{2.5} and \eqref{bolr}, it is sufficient to have the claimed commutation relation of 
the slash operator with 
the differential operators $R_k$ and $L_k$. To prove this we follow the proof of \cite[Lemma 5.2 (i)]{Ono1} 
(for detailed calculation, see \cite[Theorem 4.12 (i)]{RKS}). The only difference is that 
$\alpha \in \mathrm{GL}_2^+(\mathbb{Q})$ but   
the power of $\det \alpha$ will get balanced from both the sides. 
\end{proof}

\begin{lemma}
If $f\in H^!_k(N,\chi)$ then $f|_k\omega(N) \in H^!_k(N,\overline{\chi})$.
\label{lemma 4.2}
\begin{proof}
Recall that $\omega(N)=\begin{pmatrix}0&-1\\N&0\end{pmatrix}$. For any $\gamma=\begin{pmatrix}a&b\\cN&d\end{pmatrix}\in\Gamma_0(N)$, we have 
\[
(f|_k \omega(N))|_k\gamma=f|_k \begin{pmatrix}d & -c \\ -b N & a\end{pmatrix}|_k \omega(N)
=\chi(a) f|_k\omega(N)=\bar{\chi}(\gamma) f|_k\omega(N).
\]
Next, by using Lemma \ref{lemma:commutation} we have
\[
\Delta_k(f|_k \omega(N)) = \Delta_k(f)|_k \omega(N)=0.    
\]
Finally, by using Remark \ref{lemma 4.1} we get the required cusp conditions.
\end{proof}
\end{lemma}

\section{The space of harmonic Maass forms of polynomial growth $H_k^{\#}(N,\chi)$}

Throughout this section $k, N \in \mathbb{Z}$ with $N \geq 1$ and $\chi$ denote a Dirichlet character modulo $N$ such that $\chi(-1)=(-1)^k$. In order to attach Dirichlet series to harmonic Maass forms of manageable growth, we consider a further subspace 
of $H_k^!(N, \chi)$ consisting of those forms $f \in H_k^!(N, \chi)$ which have at most polynomial growth at any cusp of 
$\Gamma_0(N)$, that is, the Fourier expansions have the shape
\begin{equation}
f(\tau)=\sum_{n=0}^{\infty}c_f^+(n)q^n+c_f^-(0)v^{1-k}+\sum_{n<0}c^-_f(n)\Gamma(1-k,-4\pi nv)q^n,
\label{3.1}
\end{equation}
at the cusp $i\infty$ and analogous expansions at all other cusps of $\Gamma_0(N)$. We denote the subspace of all such forms by $H^{\#}_k(N,\chi)$ (note that the same notation has been used in \cite[Eq. 4.9]{Ono1} for a slightly different subspace) and call them {\em harmonic Maass forms of polynomial growth}. In this section, our aim is to study the space $H_k^{\#}(N,\chi)$ and construct examples in this space.  
\begin{prop}\label{prop:hashproperty}
Suppose $k \neq 1$ and $f \in H_k^{\#}(N,\chi)$. Then we have the following observations.
\begin{enumerate}
\item[(i)] $f|_k\omega(N)\in H_k^{\#}(N,\overline{\chi})$.

\item [(ii)] $\xi_k(f)\in M_{2-k}(N,\overline{\chi})$.

\item [(iii)] If $k\leq 0$ then $D^{1-k}(f)\in M_{2-k}(N,\chi)$.

\item [(iv)] Suppose $k \leq 0$ and $f$ has Fourier series expansion \eqref{3.1} then 
$c_f^{\pm}(n)=O(1), n \in \mathbb{Z}$.

\item [(v)] For $k > 2$, $H_k^{\#}(N, \chi)=M_{k}(N, \chi)$.
\end{enumerate}
\end{prop}
\begin{proof}
By using Lemma \ref{lemma 4.1} and Lemma \ref{lemma 4.2}, we have $(i)$.
We have $\xi_k(f) \in M_{2-k}^!(N,\overline{\chi})$ and $D^{1-k}(f) \in M_{2-k}^!(N,\chi)$. Since $\xi_k$ and $D^{1-k}$ commutes with the slash operator, by using \eqref{2.7} and \eqref{eq:DFourier} we have $(ii)$ and $(iii)$.
By using $(ii), (iii)$ together with \eqref{2.7}, \eqref{eq:DFourier} and the bound for coefficients of modular forms,
we have $(iv)$. 
Since $M_{2-k}(N, \overline{\chi})=\{0\}$ for $k> 2$, by using \eqref{2.7} we have $(v)$.
\end{proof} 

\begin{Rem}\label{Rem:signk}
\begin{enumerate}
\item[(i)] In view of the above Proposition, the only interesting case for our purpose will be 
$k\leq 0$ and $k=2$.
\item[(ii)] Suppose $k \leq 0$. If $f\in H_k^{\#}(N, \chi)$ with Fourier series expansion \eqref{3.1} then the Dirichlet series $L^{\pm}(f, s)$ attached to $f$ are absolutely convergent in the half-plane {\rm Re}$(s)>1$.  
\end{enumerate}
\end{Rem}

\subsection{Examples} Now we construct examples in the space $H_k^{\#}(N,\chi)$ for $k$ negative integer. 
We begin by recalling the basic facts about Eisenstein series of integral weight. We refer the reader to \cite[Chapter 8]{CS} for details. 
Let $k$ be a {\em negative integer}. 
Let $\rho$ be a cusp of $\Gamma_0(N)$ and $\gamma_\rho\in\mathrm{SL}_2(\mathbb{Z})$ such that 
$\gamma_\rho(i\infty)=\rho$. 
Put $\Gamma_\rho=\Gamma_0(N)\cap \gamma_\rho \Gamma_{\infty} \gamma_\rho^{-1}$ where 
$\Gamma_{\infty}$ is the stabiliser of 
$i\infty$ in 
$\mathrm{SL}_2(\mathbb{Z})$. Assume that $\chi$ is trivial on $\Gamma_\rho$. 
We define the Eisenstein series corresponding to the cusp ${\rho}$ as follows. 
\begin{equation*}
E_{2-k,\rho}(N,\chi;\tau)=\sum_{g \in \Gamma_{\rho} \backslash \Gamma_0(N)} \overline{\chi(g)} 
j\left(\gamma_{\rho}^{-1} g, \tau\right)^{k-2},
\label{3.2}
\end{equation*}
where $j(\gamma,\tau)=c\tau+d$ for $\gamma=\left(\begin{smallmatrix}a&b\\c&d\end{smallmatrix}\right)$. 
Then we have 
the following result.
\begin{thm}\emph{(Theorem 8.2.3, \cite{CS})}
The Eisenstein series $E_{2-k,\rho}(N,\chi;\tau)\in M_{2-k}(N, \chi)$. Moreover, $E_{2-k,\rho}(N,\chi;\tau)$ vanishes at every other cusp of $\Gamma_0(N)$ except at $\rho$ where it is $1$.
\label{thm 3.4}
\end{thm} 
We construct the preimage of this Eisenstein series under the shadow operator $\xi_k$. Consider the following function. 
\begin{equation}\label{eq:defF}
\mathcal{F}_{k,\rho}(N,\chi;\tau)=\sum_{g \in \Gamma_{\rho} \backslash \Gamma_0(N)} \overline{\chi(g)} \left.\frac{v^{1-k}}{1-k}\right|_k \gamma_{\rho}^{-1} g (\tau), \ 
\tau = u+iv \in \mathbb{H}.
\end{equation} 
This series converges absolutely for any negative integer $k$.
Moreover, we prove the following theorem.
\begin{thm}
For any negative integer $k$, the function $\mathcal{F}_{k,\rho}(N,\chi;\tau)\in H^{!}_k(N,\chi)$ and has 
at most polynomial growth at any cusp of $\Gamma_0(N)$. 
In particular, $\mathcal{F}_{k,\rho}(N,\chi;\tau)\in H^{\#}_k(N,\chi)$ with the shadow 
$\xi_{k}(\mathcal{F}_{k,\rho}(N,\chi;\tau))={E}_{2-k,\rho}(N,\overline{\chi};\tau)$. 
\label{thm 3.6}
\begin{proof}
Since the series considered in \eqref{eq:defF} is absolutely convergent, it is routine to check that the function 
$\mathcal{F}_{k,\rho}(N,\chi;\tau)$ satisfies modularity with character $\chi$ for the group $\Gamma_0(N)$. We have
\[
\xi_k\left(\overline{\chi(g)}\frac{v^{1-k}}{1-k}\right)=\chi(g).
\]
Since the series $\sum_{g \in \Gamma_{\rho} \backslash \Gamma_0(N)} j(\gamma_\rho^{-1}g, \tau)^{k-2}$ is 
uniformly convergent on any compact subset
of $\mathbb{H}$, by using Lemma \ref{lemma:commutation} we have  
$\xi_{k}(\mathcal{F}_{k,\rho}(N,\chi;\tau))={E}_{2-k, \rho}(N,\overline{\chi};\tau)$. Next, since 
$\xi_{k}(\mathcal{F}_{k, \rho}(N,\chi;\tau))$ is holomorphic, it follows that 
$$\Delta_k(\mathcal{F}_{k, \rho}
(N,\chi;\tau))=\xi_{2-k}\left(\xi_k(\mathcal{F}_{k, \rho}(N,\chi;\tau))\right)=0.$$ 
Finally, we verify the cusp conditions. Let $\nu$ be any cusp of $\Gamma_0(N)$ with 
$\gamma_{\nu}(i\infty)=\nu$. We have 
\[
\begin{split}
\mathcal{F}_{k, \rho}(N,\chi;\tau)|_k\gamma_{\nu}
=\frac{v^{1-k}}{1-k}\sum_{g \in \Gamma_{\rho}} \backslash \Gamma_0(N) \overline{\chi(g)} j(\gamma_{\rho}^{-1} g\gamma_{\nu}, \tau)^{-k}|j(\gamma_{\rho}^{-1} g\gamma_{\nu}, \tau)|^{2k-2}.
\end{split}
\]
Therefore we have 
\[
\begin{split}
\left|\mathcal{F}_{k,\rho}(N,\chi;\tau)|_k\gamma_{\nu}(\tau)\right|&\leq \frac{v^{1-k}}{1-k}\sum_{g \in \Gamma_{\rho} \backslash \Gamma_0(N)}|j(\gamma_{\rho}^{-1} g\gamma_{\nu}, \tau)|^{k-2}\\&=\frac{v^{1-k}}{1-k}\sum_{h=\left(\begin{smallmatrix}
a&b \\
c&d
\end{smallmatrix}\right) \in \gamma^{-1}_\rho \Gamma_{\rho} \gamma_\rho \backslash \gamma^{-1}_\rho \Gamma_0(N) \gamma_{\nu}}|c \tau+d|^{k-2}.
\end{split}
\]
In the last line we used the fact that the correspondence $g \mapsto \gamma_\rho^{-1} g \gamma_\nu$ is a bijection between 
$\Gamma_\rho \backslash \Gamma_0(N)$ and  $\gamma^{-1}_\rho \Gamma_{\rho} \gamma_\rho \backslash \gamma^{-1}_\rho \Gamma_0(N) \gamma_{\nu}$.
We now split the sum in two sums, one with $c=0$ and the other one with $c\neq 0$. By definition 
$\gamma^{-1}_\rho \Gamma_{\rho} \gamma_\rho=\Gamma_{\infty} \cap \gamma^{-1}_\rho \Gamma_0(N) \gamma_\rho$.
We have 
\[
\left|\mathcal{F}_{k,\rho}(N,\chi;,\tau)|_k\gamma_{\nu}(\tau)\right|\leq \frac{v^{1-k}}{1-k}\left[\sum_{\substack{\left(\begin{smallmatrix}
a&b \\
c&d
\end{smallmatrix}\right) \in \gamma^{-1}_\rho \Gamma_\rho \gamma_\rho \backslash \gamma^{-1}_\rho\Gamma_0(N) \gamma_\nu \\ c=0}}1+\sum_{\substack{\left(\begin{smallmatrix}
a&b \\
c&d
\end{smallmatrix}\right) \in \gamma^{-1}_\rho \Gamma_\rho \gamma_\rho \backslash \gamma^{-1}_\rho \Gamma_0(N) \gamma_{\nu}\\c\neq 0}}|c \tau+d|^{k-2}\right]
\]
Now, if $\nu$ is not equivalent to $\rho$ modulo $\Gamma_0(N)$ then $g \gamma_{\nu}(i \infty) \neq 
\gamma_\rho(i \infty)$ for any $g \in \Gamma_0(N)$, which implies that $\gamma^{-1}_\rho g \gamma_{\nu} \notin \Gamma_{\infty}$, that is, $\Gamma_{\infty} \cap \gamma^{-1}_\rho \Gamma_0(N) \gamma_{\nu}=\emptyset$ and hence the first term in the above sum is zero. On the other hand if $\nu$ is 
$\Gamma_0(N)$-equivalent to 
$\rho$ then we choose $\gamma_{\nu}=\gamma_\rho$ and hence the first sum will have only one term 
equal to $1$. 
Thus we have 
\[
\left|\mathcal{F}_{k,\rho}(N,\chi;\tau)|_k\gamma_{\nu}(\tau)\right|\leq \frac{v^{1-k}}{1-k}\left[\delta_{\rho,\nu}+\sum_{\substack{\left(\begin{smallmatrix}
a&b \\
c&d
\end{smallmatrix}\right) \in \gamma^{-1}_\rho \Gamma_{\rho} \gamma_\rho \backslash \gamma^{-1}_\rho 
\Gamma_0(N) \gamma_{\nu}\\c\neq 0}}|c \tau+d|^{k-2}\right],
\]
where $\delta_{\rho,\nu}$ is equal to $1$ if $\nu$ is $\Gamma_0(N)$-equivalent to $\rho$ and $0$ otherwise. 
Now by using the bound given in \cite[Theorem 5.1.1]{Ra}, we get a constant $C$ such that 
\[
\left|\mathcal{F}_{k,\rho}(N,\chi;\tau)|_k\gamma_{\nu}(\tau)\right|\leq\frac{v^{1-k}}{1-k}\left[\delta_{\rho, \nu}+
C(|\tau|^{k-2}+|\tau|^{\frac{k-2}{2}})\right].
\]
This implies that
\begin{equation*}
\mathcal{F}_{k,\rho}(N,\chi;\tau)|\gamma_{\nu}(\tau)=O(\delta_{\rho,\nu}v^{1-k}+v^{-k/2}) \ \ {\rm for}\ v \rightarrow \infty.
\label{gr}
\end{equation*}
This implies that $\mathcal{F}_{k, \rho}(N,\chi;\tau) \in H_k^!(N,\chi)$ having at most polynomial growth at 
any cusp of 
$\Gamma_0(N)$. 
\end{proof}
\end{thm} 

\begin{Rem}
The harmonic Maass forms constructed in \cite{HP} whose shadows are more general Eisenstein series than ours at the cusp $i\infty$ are also examples of the harmonic Maass forms of polynomial growth. In fact, our example corresponding to the cusp $\rho=i\infty$ is a special case of the examples constructed in \cite{HP}.
Note that the approach adopted in Theorem \ref{thm 3.6} can be easily generalized to any  
congruence subgroup with a multiplier system whereas we restrict ourselves to remain 
in the context of the paper.  
\end{Rem}

Define $\mathcal{E}_{2-k}(N,\chi)$ as the $\mathbb{C}$-span of all Eisenstein series of level $N$, character $\chi$ and weight $2-k$ at the different cusps of $\Gamma_0(N)$. We have the following result.
\begin{prop}\emph{(Proposition 8.5.15, \cite{CS})}
For any negative integer $k$, the dimension of the space $\mathcal{E}_{2-k}\left(N, \chi\right)$ is equal to
$$
e(N, \chi)= \sum_{\substack{C \mid N\\\text{gcd}(C, N / C) \mid N / m_{\chi}}} \phi(\operatorname{gcd}(C, N / C))
$$
where $m_{\chi}$ is the conductor of $\chi$ and $\phi$ denotes Euler's phi function. If $\chi$ is trivial then the dimension of $\mathcal{E}_{2-k}(N,\chi)$ 
is equal to the number of $\Gamma_0(N)$-inequivalent cusps. 
\label{prop 3.5}
\end{prop}
Denote by $\mathcal{E}^{\#}_k(N,\chi)\subseteq H^{\#}_k(N,\chi)$, the $\mathbb{C}$-span of the set of functions 
$\mathcal{F}_{k,\rho}(N,\chi;\tau)$ as ${\rho}$ varies over the $\Gamma_0(N)$-inequivalent cusps. We have the following corollary.
\begin{cor}\label{cor:1}
The restriction of the shadow operator $\xi_k$ to $\mathcal{E}^{\#}_k(N,\chi)$ is an isomorphism from
$\mathcal{E}^{\#}_k(N,\chi)$ to $\mathcal{E}_{2-k}(N,\overline{\chi})$. Moreover, we have 
\[
\text{dim}\ \mathcal{E}^{\#}_k(N,\chi)=\sum_{\substack{C \mid N\\\text{gcd}(C, N / C) \mid N / m_{\overline{\chi}}}} \phi(\operatorname{gcd}(C, N / C)),
\] 
where $m_{\overline{\chi}}$ is the conductor of $\overline{\chi}$.
\begin{proof}
Indeed, Theorem \ref{thm 3.6} shows that the above restriction of the shadow map is surjective. 
By using \eqref{2.7} and Lemma \ref{lemma:commutation}, we see that the kernel of the restricted 
shadow map is $M_k(N,\chi)$ (see \cite[Eq. 5.12]{Ono1}). For $k$ negative, $M_k(N,\chi) = \{0\}$. 
\end{proof}
\end{cor}

\section{Intermediate results}

The aim of this section is to prepare ourselves for the next two sections by establishing some lemmas and Theorem \ref{thm 4.8} which can be considered as an analogue of Hecke's converse theorem for 
the space 
$H_k^\#(N, \chi)$. Some of these intermediate results might be of independent interest as well.     
\begin{lemma}
Let $k \leq 0$. Suppose 
\[
f(\tau)=\sum_{n=0}^{\infty}c_f^+(n)q^n+c_f^-(0)v^{1-k}+\sum_{n<0}c^-_f(n)\Gamma(1-k,-4\pi nv)q^n
\]
with $c_f^{\pm}(n)$ bounded by $O\left(|n|^{\sigma}\right)$ for some $\sigma \geq 0$. Then the series in the right-hand side is convergent absolutely and uniformly on any compact subset of $\mathbb{H}$, and defines a real analytic function on 
$\mathbb{H}$. Moreover, we have 
\begin{gather*}
f(\tau)=O(v^{-\sigma-1+k})~~~\text{as}~v\rightarrow 0,\\
f(\tau)-c_f^+(0)-c_f^-(0)v^{1-k} =O\left(e^{-2 \pi v}\right)~~~\text{as}~v\rightarrow \infty,
\end{gather*}
uniformly in $\mathrm{Re}(\tau)$.
\label{lemma 4.7}
\begin{proof}
To prove this lemma we make appropriate changes in the proof of \cite[Lemma 4.3.3]{miyake} as we deal with real analytic functions in place of holomorphic functions.  
By using Euler-Gauss formula \cite[Eq. 3.2.9]{miyake}, we have
\begin{equation}\label{eq:euler-gauss}
\lim_{n\to\infty}\frac{n^{\sigma}}{(-1)^n{-\sigma-1\choose n}}=\Gamma(\sigma+1),
\end{equation}
where 
$$
{-\sigma-1\choose n} =\frac{(-\sigma-1) ((-\sigma-1)-1) ((-\sigma-1)-2) \ldots ((-\sigma-1)-(n-1))}{n!}.
$$
Since $c_f^+(n) = O(n^\sigma), \ n \geq 1$, there exists a constant $L > 0$ such that 
\[
|c_f^+(n)|\leq L (-1)^n {-\sigma-1\choose n}, ~~~n \geq 1.
\]
Since $k \leq 0$, $1-k \geq 1$ and therefore by using \eqref{eq:gammasum}, for any $n\geq 1$ we have 
\[
\begin{split}
\Gamma(1-k,4\pi nv) & =\Gamma(1-k) e^{-4\pi n v} \sum_{l=0}^{-k} \frac{(4\pi n v)^l}{l!}\\
&= \left\{\begin{array}{cc} O(e^{-4\pi n v}), \ \ & {\rm if} \ v \rightarrow 0, \\ O(v^{-k} e^{-4 \pi n v}), \ \ & {\rm if} \ v \rightarrow \infty. \end{array}\right.
\end{split} 
\]
Since $c_f^-(-n) = O(n^\sigma)$, by using \eqref{eq:euler-gauss} once again, we have 
\[
|c_f^-(-n) \Gamma(1-k,4\pi nv)| \leq  L' (-1)^n e^{-4 \pi n v} \sum_{l=0}^{-k} \frac{(4\pi v)^l}{l!} {-\sigma- l -1\choose n},
\]
for some constant $L'$. Thus we have
\begin{equation}\label{eq:fdecay}
\begin{split}
|f(\tau)|  & \leq L \sum_{n=0}^{\infty} (-1)^n {-\sigma-1\choose n} e^{-2\pi n v} + O (v^{1-k})\\ 
&\hspace*{6cm}+ L'  
\sum_{l=0}^{-k} \frac{(4\pi v)^l}{l!}\sum_{n=1}^{\infty} (-1)^n {-\sigma-l-1\choose n}e^{-2\pi n v}\\
& \leq L (1-e^{-2\pi v})^{-\sigma-1} + O(v^{1-k}) + L' \sum_{l=0}^{-k} \frac{(4\pi v)^l}{l!} ((1-e^{-2\pi v})^{-\sigma -l-1} -1). 
\end{split}
\end{equation}
This estimate implies that the given series $f(\tau)$ is convergent absolutely and uniformly on any compact subset 
of $\mathbb{H}$. Also, from \eqref{2.7} we have
\begin{gather*}
\xi_k(c_f^+(n) q^n) = 0, \ \ \xi_k(c_f^-(0) v^{1-k}) = (1-k)\overline{c_f^{-}(0)},\\
\xi_k(c^-_f(-n)\Gamma(1-k, 4\pi n v)q^{-n}) = -(4\pi)^{1-k} \overline{c_f^{-}(-n)} n^{1-k} q^n, \ \ n \geq 1.
\end{gather*}
By using the bound $c_f^{-}(-n) = O(n^\sigma)$, we get that the series 
$$
(1-k)\overline{c_f^{-}(0)} -(4\pi)^{1-k} \sum_{n \geq 1} \overline{c_f^{-}(-n)} n^{1-k} q^n
$$ 
is uniformly convergent on any compact subset of $\mathbb{H}$ and defines a holomorphic function.
Moreover, by using \eqref{2.5} and \eqref{2.7} we have
\[
\Delta_k (f(\tau)) = -\xi_{2-k} (\xi_k(f(\tau))) = 0.
\] 
Since solutions of elliptic partial differential equations with real analytic coefficients, such as $\Delta_k(F) = 0$, are real analytic 
(cf. \cite[Page 57]{FJ}), therefore $f(\tau)$ defines a real analytic function on $\mathbb{H}$.

Since $1-k \geq 1$ and $(1-e^{-2\pi v})^{-\sigma-1} = O(v^{-\sigma-1})$ as $v \rightarrow 0$, by using \eqref{eq:fdecay} we have $|f(\tau)| = v^{-\sigma-1+k}$ as $v \rightarrow 0$.

Next, we have
\[
f(\tau)-c_f^+(0)-c_f^-(0)v^{1-k} =  e^{2\pi i \tau} g(\tau),\]
where
\[
g(\tau) = \sum_{n=0}^{\infty}c_f^+(n+1)q^n+\sum_{n\leq -2}c^-_f(n+1)\Gamma(1-k,-4\pi (n+1) v) q^n.
\]
By using the asymptotic formula given in \eqref{gamma},
for any $n \leq -2$ we have
\[
|\Gamma(1-k, -4\pi(n+1) v) q^n| \leq (4\pi (-n-1) v)^{-k} e^{2\pi (n+2) v} \ \ {\rm as}\ v\ \rightarrow \infty.
\]
Now by following the calculations done to obtain \eqref{eq:fdecay}, we get that $g(\tau)$ is bounded in a neighborhood of
$i \infty$. Thus we have
\[
f(\tau)-c_f^+(0)-c_f^-(0)v^{1-k} = O(e^{-2\pi v}) \ \ {\rm as} \ v \rightarrow \infty.
\]
\end{proof}
\end{lemma}

\begin{lemma}
Let $k \in \mathbb{Z}$. 
Let $f, g$ be any two real analytic functions on $\mathbb{H}$ such that $g=f|_k\omega(N)$. Let
\begin{equation*}
H(\tau)=2iv\frac{\partial f}{\partial u}(\tau)+kf(\tau),~~I(\tau)=2iv\frac{\partial g}{\partial u}(\tau)+kg(\tau), \ \ 
\tau = u+iv \in \mathbb{H}.
\label{4.7}
\end{equation*}
Then for any $t>0$ we have
\[
H|_k\omega(N) (i t) = N^{-k/2} (i t)^{-k} H\left(-\frac{1}{N (it)}\right)=-I(it).
\]
\label{lemma 4.6}
\begin{proof}
We have 
\[
g(\tau)=N^{-k / 2}\tau^{-k} f\left(-\frac{1}{N \tau}\right).
\]
Differentiating both sides with respect to $u$, we get 
\[
\frac{\partial g}{\partial u}(\tau)=-kN^{-k / 2}\frac{\tau^{-k}}{\tau} f\left(-\frac{1}{N \tau}\right)+N^{-k / 2}\tau^{-k}\frac{\partial f}{\partial u}\left(-\frac{1}{N \tau}\right)\left(\frac{1}{N\tau^2}\right).
\]
Multiplying both sides by $2iv$ and evaluating at $\tau=it$, we get 
\[
\begin{split}
2it\frac{\partial g}{\partial u}(it)&+kN^{-k/2}i^{-k}t^{-k}f\left(\frac{i}{Nt}\right)=-kN^{-k/2}i^{-k}t^{-k}f\left(\frac{i}{Nt}\right)\\&\hspace{5.5cm}-2itN^{-k/2}i^{-k}t^{-k}\left(\frac{1}{Nt^2}\right)\frac{\partial f}{\partial u}\left(\frac{i}{Nt}\right)\\&\implies I(it)=-N^{-k / 2}i^{-k}t^{-k}\left(2i\text{Im}\left(\frac{i}{Nt}\right)\frac{\partial f}{\partial u}\left(\frac{i}{Nt}\right)+kf\left(\frac{i}{Nt}\right)\right)\\&\implies I(it)=-N^{-k / 2}i^{-k}t^{-k}H\left(\frac{i}{Nt}\right).
\end{split}
\]
\end{proof}
\end{lemma}

\begin{lemma}
For any $\nu>0$, the function $W_{\nu}(s)$ $($defined by \eqref{eq:wnu}$)$ is an analytic function in the half-plane 
$\emph{Re}(s)>0$. Moreover, for any 
$x, \sigma > 0$ 
we have 
\begin{equation}\label{eq:lemma w}
\Gamma(\nu, 2x) e^{-x}=\frac{1}{2\pi i}\int_{\sigma-i\infty}^{\sigma+i\infty}x^{-s}W_{\nu}(s)ds.
\end{equation}
\label{lemma w} 
\begin{proof}
We have $|\Gamma(\nu, 2x)| \leq \Gamma(\nu)$ for $x \geq 0$. By using \eqref{gamma}
we get that $W_{\nu}(s)$ defines an analytic function in the half-pane $\text{Re}(s)>0$.
Since $W_{\nu}(s)$ is the Mellin transform of $f_{\nu}(x)=\Gamma(\nu,2x)e^{x}$, 
to prove \eqref{eq:lemma w} it is sufficient to check the hypothesis of Mellin inversion theorem (see \cite[Proposition 3.1.22]{CS}) for 
$f_{\nu}(x)$. Indeed this is the case as $f_{\nu}(x)$ is a continuous function and  $|f_\nu(x)| \leq \Gamma(\nu) e^{x}$ for $x \geq 0$. The latter inequality gives us that $f_\nu(x) = O(1)$ as $x \rightarrow 0$. 
By using \eqref{gamma}, we get the required growth condition as $x \rightarrow \infty$.  
\end{proof}
\end{lemma}
\begin{lemma}\label{lemma:decay}
If $\nu$ is a positive integer then for any $\mu > 0$, we have $W_\nu(s) = O({\rm Im}(s)^{-\mu})$ 
on any line ${\rm Re}(s) = \sigma, \sigma > 0,$ as $|{\rm Im}(s)| \rightarrow \infty$.
\end{lemma}
\begin{proof}
By using \eqref{eq:gammasum}, we have
\[
\begin{split}
W_\nu(s) = \Gamma(\nu) \sum_{l=0}^{\nu-1} \frac{2^l}{l!} \int_0^\infty e^{-x} x^{s+l} \frac{dx}{x}
= \Gamma(\nu) \sum_{l=0}^{\nu-1} \frac{2^l}{l!} \Gamma(s+l).\\
\end{split}
\]
Now by using Stirling's estimate for the gamma function \cite[Eq. (3.2.8)]{miyake}, the lemma follows.
\end{proof}

\begin{thm}
Let $k$ be a negative integer and $N$ be a positive integer. Let $f$ and $g$ be two functions defined on 
$\mathbb{H}$ given by the formal Fourier series \eqref{eq:fgfourier}
with $c_f^{\pm}(n), c_g^{\pm}(n)$ bounded by $O\left(|n|^{\alpha}\right), n \in \mathbb{Z}$, for some 
$\alpha \geq 0$. Then the following two statement are equivalent.
\begin{enumerate}
\item $g=f|_k\ \omega(N)$.
\item The completed Dirichlet series $\Lambda_N(f,s), \Lambda_N(g,s), \Omega_N(f,s),\Omega_N(g,s)$ satisfy condition (2) (a) of Theorem \ref{thm:main}. 

\end{enumerate}
\label{thm 4.8}
\begin{proof}
$(1)\implies (2).$ We have 
\[
\begin{split}
c_f^-(-n)\int_{0}^{\infty}\Gamma(1-k,4\pi nt/\sqrt{N})e^{-2\pi nt/\sqrt{N}} t^s \frac{dt}{t}&=\left(\frac{2\pi n}{\sqrt{N}}\right)^{-s}c_f^-(-n)\int_{0}^{\infty}\Gamma(1-k,2x)e^{-x} x^s \frac{dx}{x}\\&=\left(\frac{2\pi}{\sqrt{N}}\right)^{-s} W_{1-k}(s) \frac{c_f^-(-n)}{n^s},
\end{split} 
\]
and
\[
c_f^+(n)\int_0^{\infty}e^{-2\pi nt/\sqrt{N}}t^{s}\frac{dt}{t}=\left(\frac{2\pi}{\sqrt{N}}\right)^{-s}\Gamma(s)
\frac{c_f^+(n)}{n^s}.
\]
Now for Re$(s)>\alpha+1$, we have 
\[
\begin{split}
\int_0^{\infty}\bigg(f\left(\frac{it}{\sqrt{N}}\right)-c^+_f(0)&-\frac{c^-_f(0)}{N^{\frac{1-k}{2}}}t^{1-k}\bigg)t^s\frac{dt}{t}\\&=\int_0^{\infty}\left(\sum_{n=1}^{\infty}c_f^+(n)e^{-2\pi nt/\sqrt{N}}\right)t^{s}\frac{dt}{t}\\&+\int_{0}^{\infty}\left(\sum_{n=1}^{\infty}c_f^-(-n)\Gamma(1-k,4\pi nt/\sqrt{N})e^{2\pi nt/\sqrt{N}}\right)t^s\frac{dt}{t}\\&=\left(\frac{2\pi}{\sqrt{N}}\right)^{-s}\left[\Gamma(s) L^+(f,s)+W_{1-k}(s)L^-(f,s)\right]\\&=\Lambda_N(f,s).
\end{split}
\]
Therefore for Re$(s) > \alpha +1$, we have 
\[
\begin{split}
\Lambda_N(f,s)&=\int_0^{\infty}\bigg{(}f\left(\frac{it}{\sqrt{N}}\right)-c^+_f(0)-\frac{c^-_f(0)}{N^{\frac{1-k}{2}}}t^{1-k}\bigg{)}t^s\frac{dt}{t}\\ 
&=\int_0^{1}+\int_1^{\infty}\left(f\left(\frac{it}{\sqrt{N}}\right)-c^+_f(0)-\frac{c^-_f(0)}{N^{\frac{1-k}{2}}}t^{1-k}\right)t^s\frac{dt}{t}\\
&=\int_1^{\infty}\left(f\left(\frac{it}{\sqrt{N}}\right)-c_f^+(0)-\frac{c^-_f(0)}{N^{\frac{1-k}{2}}}t^{1-k}\right)t^s\frac{dt}{t}+\int_1^{\infty}f\left(\frac{i}{\sqrt{N}t}\right)t^{-s}\frac{dt}{t}\\
&-\frac{c_f^+(0)}{s}-\frac{c^-_f(0)}{N^{\frac{1-k}{2}}}\frac{1}{s-k+1}
\\&=\int_1^{\infty}\left(f\left(\frac{it}{\sqrt{N}}\right)-c_f^+(0)-\frac{c^-_f(0)}{N^{\frac{1-k}{2}}}t^{1-k}\right)t^s\frac{dt}{t}+i^{k}\int_1^{\infty}g\left(\frac{it}{\sqrt{N}}\right)t^{k-s}\frac{dt}{t}\\
&-\frac{c_f^+(0)}{s}-\frac{c^-_f(0)}{N^{\frac{1-k}{2}}}\frac{1}{s-k+1}.
\end{split}
\]
For Re$(s)>k$, we have  
\[
\int_{1}^{\infty}c_g^+(0)t^{k-s-1}dt=-\frac{c_g^+(0)}{k-s}.
\] 
For Re$(s) > 1$, we have
\[
\int_{1}^{\infty}\frac{c^-_g(0)}{N^{\frac{1-k}{2}}}t^{1-k}t^{k-s-1}dt=-\frac{c_g^-(0)}{N^{\frac{1-k}{2}}}\frac{1}{1-s}.
\]
Therefore for Re$(s)> \alpha+1$, we have 
\begin{equation}\label{eq:lambdaf}
\begin{split}
\Lambda^*_N(f,s)=\Lambda_N(f,s)+\frac{c_f^+(0)}{s}+&\frac{c_g^+(0) i^{k}}{k-s}+\frac{c^-_f(0)}{N^{\frac{1-k}{2}}}\frac{1}{s-k+1}+\frac{c_g^-(0)i^k}{N^{\frac{1-k}{2}}}\frac{1}{1-s}\\&=\int_{1}^{\infty} \left(f\left(\frac{i t}{\sqrt{N}}\right)-c_f^+(0)-\frac{c^-_f(0)}{N^{\frac{1-k}{2}}}t^{1-k}\right) t^{s-1} d t
\\&+i^{k} \int_{1}^{\infty}\left(g\left(\frac{i t}{\sqrt{N}}\right)-c_g^+(0)-\frac{c^-_g(0)}{N^{\frac{1-k}{2}}}t^{1-k}\right) t^{k-s-1} d t
\end{split}
\end{equation}
Since $g|_k\ {\omega(N)} = (-1)^k f$, we have 
\begin{equation}\label{eq:lambdag}
\begin{split}
\Lambda_N^*(g,s)=\Lambda_N(g,s)+\frac{c_g^+(0)}{s}+&\frac{c_f^+(0) i^{-k}}{k-s}+\frac{c^-_g(0)}{N^{\frac{1-k}{2}}}\frac{1}{s-k+1}+\frac{c_f^-(0)i^{-k}}{N^{\frac{1-k}{2}}}\frac{1}{1-s}\\&=\int_{1}^{\infty} \left(g\left(\frac{i t}{\sqrt{N}}\right)-c_g^+(0)-\frac{c^-_g(0)}{N^{\frac{1-k}{2}}}t^{1-k}\right) t^{s-1} d t
\\&+i^{-k} \int_{1}^{\infty}\left(f\left(\frac{i t}{\sqrt{N}}\right)-c_f^+(0)-\frac{c^-_f(0)}{N^{\frac{1-k}{2}}}t^{1-k}\right) t^{k-s-1} d t
\end{split}
\end{equation}
Now by using Lemma \ref{lemma 4.7} we see that the integral representations of  
$\Lambda^*_N(f,s)$ and $\Lambda^*_N(g,s)$ define analytic functions and are bounded in any vertical strip in the whole complex plane. Replacing $s$ with $k-s$ in the integral expression \eqref{eq:lambdag} of $\Lambda_N(g,s)$ and then comparing with \eqref{eq:lambdaf} we obtain
$$
\Lambda_N(f,s) = i^k \Lambda_N(g, k-s).
$$ 
Next, we have 
\[
\frac{\partial f}{\partial u}(\tau)=2\pi i \left( \sum_{n=1}^{\infty}nc_f^+(n)q^n+\sum_{n<0}nc^-_f(n)\Gamma(1-k,-4\pi nv)q^n \right)
\]
\[
\frac{\partial g}{\partial u}(\tau)=2\pi i \left(\sum_{n=1}^{\infty}nc_g^+(n)q^n+\sum_{n<0}nc^-_g(n)\Gamma(1-k,-4\pi nv)q^n\right).
\]
For any $n \geq 1$, we have
\[
\begin{split}
2\pi i n c_f^-(-n)\int_{0}^{\infty}\Gamma(1-k,4\pi nt/\sqrt{N})&e^{2\pi nt/\sqrt{N}}t^{s+1}\frac{dt}{t}\\&=2\pi i n c_f^-(-n) \left(\frac{2\pi n}{\sqrt{N}}\right)^{-s-1}\int_{0}^{\infty}\Gamma(1-k,2x)e^{x}x^{s+1}\frac{dx}{x}\\&=W_{1-k}(s+1) i\sqrt{N}\left(\frac{2\pi}{\sqrt{N}}\right)^{-s}\frac{c_f^-(-n)}{n^s}.
\end{split}
\]
\[
2\pi i n c_f^+(n)\int_0^{\infty}e^{-2\pi nt/\sqrt{N}}t^{s+1}\frac{dt}{t}=\Gamma(s+1)
i\sqrt{N}\left(\frac{2\pi}{\sqrt{N}}\right)^{-s}\frac{c_f^+(n)}{n^s}.
\]
So we get 
\[
\begin{split}
\int_{0}^{\infty}\left(H\left(\frac{it}{\sqrt{N}}\right)-kc_f^+(0)-k\frac{c^-_f(0)}{N^{\frac{1-k}{2}}}t^{1-k}\right)t^{s}\frac{dt}{t}=\frac{2i}{\sqrt{N}}\int_{0}^{\infty}\frac{\partial f}{\partial u}\left(\frac{it}{\sqrt{N}}\right)t^{s+1}\frac{dt}{t}\\+k\int_0^{\infty}\left(f\left(\frac{it}{\sqrt{N}}\right)-c^+_f(0)-\frac{c^-_f(0)}{N^{\frac{1-k}{2}}}t^{1-k}\right)t^s\frac{dt}{t}\\=-2~\Xi_N(f,s)+k\Lambda_N(f,s)=\Omega_N(f,s).
\end{split}
\]
By proceeding similar to the case of $\Lambda_N(f,s), \Lambda(g,s)$ and by using Lemma \ref{lemma 4.6}, we have
\[
\begin{split}
\Omega^*_N(f,s)=\Omega_N(f,s)+k\Bigg(\frac{c_f^+(0)}{s}-\frac{c_g^+(0) i^{k}}{k-s}&+\frac{c^-_f(0)}{N^{\frac{1-k}{2}}}\frac{1}{s-k+1}-\frac{c_g^-(0)i^k}{N^{\frac{1-k}{2}}}\frac{1}{1-s}\Bigg)\\&=\int_{1}^{\infty} \left(H\left(\frac{i t}{\sqrt{N}}\right)-kc_f^+(0) -k\frac{c^-_f(0)}{N^{\frac{1-k}{2}}}t^{1-k}\right) t^{s-1} d t \\
&-i^{k} \int_{1}^{\infty}\left(I\left(\frac{i t}{\sqrt{N}}\right)-k c_g^+(0)-k\frac{c^-_g(0)}{N^{\frac{1-k}{2}}}t^{1-k}\right) t^{k-s-1} d t.
\end{split}
\]
\[
\begin{split}
\Omega_N^*(g,s)=\Omega_N(g,s)&+k\Bigg(\frac{c_g^+(0)}{s}-\frac{c_f^+(0) i^{-k}}{k-s}+\frac{c^-_g(0)}{N^{\frac{1-k}{2}}}\frac{1}{s-k+1}-\frac{c_f^-(0)i^{-k}}{N^{\frac{1-k}{2}}}\frac{1}{1-s}\Bigg)\\&=\int_{1}^{\infty} \left(I\left(\frac{i t}{\sqrt{N}}\right)-k c_g^+(0)-k\frac{c^-_g(0)}{N^{\frac{1-k}{2}}}t^{1-k}\right) t^{s-1} d t \\
&-i^{-k} \int_{1}^{\infty}\left(H\left(\frac{i t}{\sqrt{N}}\right)-kc_f^+(0)-k\frac{c^-_f(0)}{N^{\frac{1-k}{2}}}t^{1-k}\right) t^{k-s-1} d t.
\end{split}
\]
Again by using Lemma \ref{lemma 4.7} we see that the right hand side integral representations of
$\Omega^*_N(f,s)$ and $\Omega^*_N(g,s)$ are analytic and bounded in any vertical strip in the whole complex plane.
By replacing $s$ with $k-s$ in the integral expression of $\Omega_N(g,s)$ and then comparing with $\Omega_N(f,s)$, we have
$$
\Omega_N(f,s) = -i^{k} \Omega_N(g,k-s).
$$
\\
$(2)\implies (1).$ 
By using inverse Mellin transform of $\Gamma(s)$ and Lemma \ref{lemma w},
for any $t > 0$ and $\beta_1 > \alpha + 1$, we have 
\begin{equation*}
\begin{split}
\frac{1}{2\pi i}&\int\limits_{\beta_1-i\infty}^{\beta_1+i\infty}t^{-s}\Lambda_N(f,s)ds\\&=\frac{1}{2\pi i}\left[\int\limits_{\beta_1-i\infty}^{\beta_1+i\infty}\left(\frac{2\pi t}{\sqrt{N}}\right)^{-s}\Gamma(s)L^+(f,s)ds+\int\limits_{\beta_1-i\infty}^{\beta_1+i\infty}\left(\frac{2\pi t}{\sqrt{N}}\right)^{-s}W_{1-k}(s)L^-(f,s)ds\right]\\&=\sum_{n=1}^{\infty}\left[\frac{1}{2\pi i}\int\limits_{\beta_1-i\infty}^{\beta_1+i\infty}c_f^+(n)\left(\frac{2\pi nt}{\sqrt{N}}\right)^{-s}\Gamma(s)ds+\frac{1}{2\pi i}\int\limits_{\beta_1-i\infty}^{\beta_1+i\infty}c_f^-(-n)\left(\frac{2\pi nt}{\sqrt{N}}\right)^{-s}W_{1-k}(s)ds\right]\\&=\sum_{n=1}^{\infty}\left[c^+_f(n)e^{-2\pi nt/\sqrt{N}}+c_f^-(-n)\Gamma(1-k,4\pi nt/\sqrt{N})e^{2\pi nt/\sqrt{N}}\right]\\&=f\left(\frac{it}{\sqrt{N}}\right)-c_f^+(0)-\frac{c^-_f(0)}{N^{\frac{1-k}{2}}}t^{1-k}.
\end{split}
\end{equation*}
Since $L^+(f,s)$ and $L^-(f,s)$ are bounded on $\text{Re}(s)=\beta_1$, 
by using Stirling's estimate for $\Gamma(s)$ and Lemma \ref{lemma:decay} for $W_{1-k}(s)$,
for any $\mu > 0$ we have
$$
|\Lambda_N(f,s)| = O(|\text{Im}(s)|^{-\mu}), \;\text{as} \ |\text{Im}(s)| \rightarrow \infty,
$$
on $\text{Re}(s) = \beta_1$. Now choose any real number $\beta_2$ such that $k-\beta_2 > \alpha+1$. Then by using the functional equation for $\Lambda_N(f,s)$ and by following the arguments similar to above for 
$\Lambda_N(g,s)$, for any $\mu > 0$ we have
$$
|\Lambda(f,s)| =  |\Lambda_N(g,k-s)| = O(|\text{Im}(s)|^{-\mu}), \;\text{as} \ |\text{Im}(s)| 
\rightarrow \infty,
$$
on $\text{Re}(s)=\beta_2$. Since by assumption the function
$\Lambda^*_N(f,s)=\Lambda_N(f,s)+\frac{c_f^+(0)}{s}+\frac{c_g^+(0) i^{k}}{k-s}+\frac{c^-_f(0)}{N^{\frac{1-k}{2}}}\frac{1}{s-k+1}+\frac{c_g^-(0)i^k}{N^{\frac{1-k}{2}}}\frac{1}{1-s}$ 
is holomorphic and bounded on the vertical strip $\beta_2 \leq \text{Re}(s) \leq \beta_1$,
by applying Phragm\'en--Lindel\"of theorem to $\Lambda_N^*(f,s)$, we get 
$$
|\Lambda_N^*(f,s)| = |\Lambda_N(f,s)+\frac{c_f^+(0)}{s}+\frac{i^{k}c_g^+(0)}{k-s}+\frac{c^-_f(0)}{N^{\frac{1-k}{2}}}\frac{1}{s-k+1}+\frac{i^kc_g^-(0)}{N^{\frac{1-k}{2}}}\frac{1}{1-s}|= O(|\text{Im}(s)|^{-\mu})
$$
for any $0< \mu \leq 1$ as $|\text{Im}(s)| \rightarrow \infty$, on the domain $\beta_2 \leq \text{Re}(s) \leq \beta_1$. From this, we conclude that $\Lambda(f, s)=  O(|\text{Im}(s)|^{-\mu})$, for $0<\mu\leq 1$ as 
$|\text{Im}(s)| \rightarrow \infty$. Without loss of generality we assume that $\beta_2 < k-1$. The function 
$t^{-s}\Lambda(f,s)$ has possibly simple poles at $s=0, k, 1, k-1$ with the residues $-c_f^+(0), 
i^{k}c_g^+(0)t^{-k},\frac{i^kc_g^-(0)}{N^{\frac{1-k}{2}}}t^{-1}, -\frac{c_f^-(0)}{N^{\frac{1-k}{2}}}t^{-k+1}$ respectively. Now consider a rectangular contour with vertical sides  
Re$(s)=\beta_1, \beta_2$ and horizontal sides Im$(s)=T, -T$. The integral of the function 
$t^{-s}\Lambda_N(f,s)$ along the horizontal lines $|{\rm Im}(s)| =T$ will go to $0$ as $T\rightarrow \infty$.
This will allow us to change the integration path from $\text{Re}(s)=\beta_1$ to $\text{Re}(s)=\beta_2$ to get
\[
f\left(\frac{it}{\sqrt{N}}\right)=\frac{1}{2 \pi i} \int_{\beta_2-i \infty}^{\beta_2+i \infty} \Lambda_N(f,s) t^{-s} d s+i^{k} c_g^+(0) t^{-k}+i^k\frac{c_g^-(0)}{N^{\frac{1-k}{2}}}t^{-1}.
\]
By using functional equation of $\Lambda_N(f,s)$, we get 
\[
\begin{split}
f\left(\frac{it}{\sqrt{N}}\right)&=\frac{1}{2 \pi i} \int_{\beta_2-i \infty}^{\beta_2+i \infty}i^k \Lambda_N(g,k-s) t^{-s} d s+i^{k} c_g^+(0) t^{-k}+i^k\frac{c_g^-(0)}{N^{\frac{1-k}{2}}}t^{-1}\\
&=\frac{i^kt^{-k}}{2 \pi i} \int_{k-\beta_2 - i \infty}^{k-\beta_2+i \infty} \Lambda_N(g,s) t^{s} d s+i^{k} c_g^+(0) t^{-k}+i^k\frac{c_g^-(0)}{N^{\frac{1-k}{2}}}t^{-1}\\ 
&=i^kt^{-k}g\left(\frac{i}{\sqrt{N}t}\right).
\end{split}
\] 
Since the above equality is true for any $t>0$, by replacing $t$ by $1/\sqrt{N}t$ we get 
\begin{equation}\label{eq:ftrans}
f|_k \omega(N) (it)= g(it). 
\end{equation}

By following the calculations done in the beginning, for any $t > 0$ and $\beta_1 > \alpha + 1$ we have 
\begin{equation*}
\begin{split}
\frac{1}{2\pi i}\int\limits_{\beta_1-i\infty}^{\beta_1+i\infty}t^{-s}\Xi_N(f,s)ds&=\frac{1}{2\pi i}\Bigg{[}\int\limits_{\beta_1-i\infty}^{\beta_1+i\infty}\left(\frac{2\pi t}{\sqrt{N}}\right)^{-s}\Gamma(s+1)L^+(f,s)ds\\
&\hspace{3cm}-\int\limits_{\beta_1-i\infty}^{\beta_1+i\infty}\left(\frac{2\pi t}{\sqrt{N}}\right)^{-s}W_{1-k}(s+1)L^-(f,s)ds\Bigg{]}\\
&=\sum_{n=1}^{\infty}\Bigg{[}\frac{1}{2\pi i}\int\limits_{\beta_1-i\infty}^{\beta_1+i\infty}c_f^+(n)\left(\frac{2\pi nt}{\sqrt{N}}\right)\left(\frac{2\pi nt}{\sqrt{N}}\right)^{-s-1}\Gamma(s+1)ds\\
&\hspace{1.5cm}-\frac{1}{2\pi i}\int\limits_{\beta_1-i\infty}^{\beta_1+i\infty}c_f^-(-n)\left(\frac{2\pi nt}{\sqrt{N}}\right)\left(\frac{2\pi nt}{\sqrt{N}}\right)^{-s-1}W_{1-k}(s+1)ds\Bigg{]}\\
&=\frac{t}{i\sqrt{N}}\sum_{n=1}^{\infty}2\pi i n \left[c^+_f(n)e^{-2\pi nt/\sqrt{N}}-c_f^-(-n)\Gamma(k-1,4\pi nt/\sqrt{N})e^{2\pi nt/\sqrt{N}}\right]\\
&=\frac{t}{i\sqrt{N}}\frac{\partial f}{\partial u}\left(\frac{it}{\sqrt{N}}\right).
\end{split}
\end{equation*}
Therefore we have
\[
\frac{1}{2\pi i}\int\limits_{\beta_1-i\infty}^{\beta_1+i\infty}t^{-s}\Omega_N(f,s)ds=-\frac{2t}{i\sqrt{N}}\frac{\partial f}{\partial u}\left(\frac{it}{\sqrt{N}}\right)+kf\left(\frac{it}{\sqrt{N}}\right)-kc_f^+(0)-k\frac{c^-_f(0)}{N^{\frac{1-k}{2}}}t^{1-k}.
\]
By following all the above calculations done in $\Lambda_N(f,s)$ case and by using the functional equation for 
$\Omega_N(f,s)$ we get 
\[
\frac{2i}{N t}\frac{\partial f}{\partial u}\left(\frac{-1}{N (it)}\right)+kf\left(\frac{-1}{N(it)}\right)=-N^{k/2}i^kt^k\left[2it\frac{\partial g}{\partial u}\left(it\right)+kg\left(it\right)\right],
\]
that is, 
\begin{equation}\label{eq:Htrans}
H|_k \omega (N) (it) = -I(it).
\end{equation}

Since $\Delta_k = -\xi_{2-k}\xi_k$, by applying $\Delta_k$ on the given Fourier series expansion of $f$ and $g$ we get that $\Delta_k(f)=0=\Delta_k(g)$. Put $G:=f|_k\omega(N)-g$. Then we have 
\[
\Delta_k(G)= 0.
\]
From the theory of differential equations, it is well known (cf. \cite[Page 57]{FJ}) that solutions of elliptic partial differential equations with real analytic coefficients, such as $\Delta_k(F) = 0$, are real analytic. So $G(\tau)$ is real analytic and we write 
\[
G(\tau)=G(u+iv)=\sum_{n=0}^{\infty}G_n(v)u^n.
\]
Since $G(\tau)$ satisfies $\Delta_k(G) = 0$, we have
\[
-v^2((n+2)(n+1)G_{n+2}(v)+G_n^{''} (v))+i k v((n+1)G_{n+1}(v)+iG_n^{'}(v)) = 0, \ {\rm for}\ n \geq 0.
\]
Therefore we get the following recursive formula for $G_n$.
\begin{equation}\label{eq:Grecursive}
G_{n+2}(v)=\frac{ikv(n+1)G_{n+1}(v)-kvG_n'(v)-v^2G_n''(v)}{v^2(n+1)(n+2)}, \ {\rm for}\ n \geq 0.
\end{equation}
Also, we have
\[
\begin{split}
G_0(t)=G(it)=N^{-k/2}i^{-k}t^{-k}f\left(-\frac{1}{Nit}\right)-g(it),\\
G_1(t)=\frac{\partial G}{\partial u}(it)=\frac{\partial }{\partial u}\left.\left(N^{-k/2}\tau^{-k}f\left(\frac{-1}{N\tau}\right)\right)\right|_{\tau=it}-\frac{\partial g}{\partial u}(it).
\end{split}
\]
By using \eqref{eq:ftrans} we get $G_0(t)=0$. Next, by following the steps in the proof of Lemma \ref{lemma 4.6}
and by using \eqref{eq:Htrans} we get $G_1(t)=0$. 
Therefore by using \eqref{eq:Grecursive} we have $G_n(v)=0$ for all $n\geq 0$ and hence $G(\tau)=0$, which implies that $f|_k\omega(N)(\tau)=g(\tau)$. 
\end{proof}
\end{thm}
 
\section{Proof of Theorem \ref{thm:main}}

Throughout this section, let $k$ be a negative integer and $N$ be a positive integer. 
Let $f$ and $g$ be two functions defined on 
$\mathbb{H}$ given by the formal Fourier series \eqref{eq:fgfourier}
with $c_f^{\pm}(n), c_g^{\pm}(n)$ bounded by $O\left(|n|^{\alpha}\right), n \in \mathbb{Z}$, for some 
$\alpha \geq 0$. Let $\psi$ be a non-trivial primitive Dirichlet character with conductor $m_\psi > 1$ 
and $f_\psi$ be the twisted Fourier series 
\eqref{eq:ftwistedfourier} of $f$ by $\psi$. We have 
the following twisted Dirichlet series attached to $f$.
\begin{equation*}
\begin{split}
L^+(f,s,\psi)=L^+(f_\psi, s)=\sum_{n=1}^{\infty}\frac{\psi(n)c^+_f(n)}{n^s},~~~ L^-(f,s,\psi)=L^-(f_\psi, s)=\sum_{n=1}^{\infty}\frac{\psi(n)c^-_f(-n)}{n^s}.\\
\Lambda_{N}(f,s, \psi)=\Lambda_N(f_\psi, s)=\left(\frac{m_\psi\sqrt{N}}{2 \pi}\right)^{s} \left[\Gamma(s) L^+(f,s,\psi)+W_{1-k}(s)L^-(f,s,\psi)\right].\\
\Xi_N(f,s,\psi)=\Xi_n(f_\psi,s)=\left(\frac{m_\psi\sqrt{N}}{2 \pi}\right)^{s} \left[\Gamma(s+1) L^+(f,s,\psi)-W_{1-k}(s+1)L^-(f,s,\psi)\right].\\
\Omega_N(f,s,\psi)=\Omega_N(f_\psi,s)=-2~\Xi_N(f,s,\psi)+k\Lambda_N(f,s,\psi).
\end{split}
\label{5.3}
\end{equation*}
Similarly, we have the twisted Fourier series $g_\psi$ and we attach the Dirichlet series 
$\Lambda_N(g, s, \psi)$, $\Xi_N(g, s, \psi)$ and $\Omega_N(g,s,\psi)$ to $g$.
\begin{prop}
Let all the notations be as above. Let $C_\psi$ be a constant which may depend on $\psi$. 
Then the following two statements are equivalent.
\begin{enumerate}
\item $f_{\psi}|_k \omega(N m_\psi^{2})={C_\psi} g_{\bar{\psi}}.$
\item $\Lambda_{N}(f,s,\psi),\Lambda_{N}(g,s,\bar{\psi})$  and $\Omega_{N}(f,s,{\psi}),\Omega_{N}(g,s,\bar{\psi})$  
can be analytically continued to the whole s-plane, bounded on any vertical strip, and satisfy the following functional equations
\begin{equation*}
\begin{split}
\Lambda_{N}(f,s,\psi)=i^{k} C_{\psi} \Lambda_{N}(g,k-s ,\bar{\psi}),\\
\Omega_{N}(f,s,\psi)=-i^{k} C_{\psi} \Omega_{N}(g,k-s ,\bar{\psi}).
\end{split}
\label{5.5}
\end{equation*}
\label{prop 5.1(2)}
\end{enumerate}
\label{prop 5.1}
\begin{proof}
This proposition follows as an application of Theorem \ref{thm 4.8} by putting $f=f_{\psi}, g=C_{\psi}g_{\overline{\psi}}$ and $N=Nm_\psi^2$ in the theorem.
\end{proof}
\end{prop}
\begin{lemma}
Let $f$ and $\psi$ be as above.
For any $r\in\mathbb{R}$, let 
\[
T^r=\begin{pmatrix}
1&r\\0&1
\end{pmatrix}.
\]
Then we have
\[
f_{\psi}=\tau(\overline{\psi})^{-1} \sum_{u=1}^{m_\psi} \overline{\psi}(u)\left(\left.f\right|_k T^{u / m_\psi}\right),
\]
where $\tau(\overline{\psi}) = \sum_{a=1}^{m_\psi} \overline{\psi}(a) e^{2 \pi i a/m}$ is the Gauss sum of 
$\overline{\psi}$.
\label{lemma 5.2}
\begin{proof}
Put $m = m_\psi$. We have 
\[
\left(\left.f\right|_k T^{u / m}\right)(\tau)=\sum_{n=0}^{\infty}e^{2\pi inu/m}c_f^+(n)q^n+c_f^-(0)v^{1-k}+\sum_{n<0}e^{2\pi i n u / m}c^-_f(n)\Gamma(1-k,-4\pi n v) q^n.
\]
From \cite[Lemma 3.1.1 (1)]{miyake} we have
\[
\sum_{a=1}^m \overline{\psi}(a) e^{2 \pi i a n/m} = \psi(n) \tau(\overline{\psi}),
\] 
for any integer $n$. By using this fact we get
\[
\begin{split}
\sum_{u=1}^{m} \overline{\psi}(u)\left(\left.f\right|_k T^{u / m}\right)(\tau)&=\sum_{n=0}^{\infty}\left(\sum_{u=1}^{m} \overline{\psi}(u)e^{2\pi inu/m}\right)c_f^+(n)q^n+c_f^-(0)v^{1-k}\sum_{u=1}^{m} \overline{\psi}(u)\\
&\hspace{1cm}+\sum_{n<0}\left(\sum_{u=1}^{m} \overline{\psi}(u)e^{2\pi inu/m}\right)c^-_f(n)\Gamma(1-k,-4\pi nv)q^n\\
&=\tau(\overline{\psi})\Bigg{[}\sum_{n=0}^{\infty}\psi(n)c_f^+(n)q^n+\sum_{n<0}\psi(n)c^-_f(n)\Gamma(1-k,-4\pi nv)q^n\Bigg{]}\\&=\tau(\overline{\psi})f_{\psi}(\tau).
\end{split}
\]
\end{proof}
\end{lemma}

\subsection{Proof of Theorem \ref{thm:main}: $(1) \implies (2)$}

\begin{prop}\label{prop:fpsitrans}
Let $\chi$ be a Dirichlet character modulo $N$ with conductor $m_\chi$ and  $f\in H^{\#}_k(N,\chi)$. 
Let $\psi$ be a primitive Dirichlet character of conductor $m_{\psi} >1$ as above and 
$M={\rm lcm}(N,m_{\psi}^{2},m_{\psi} m_{\chi})$. Then 
$f_{\psi}\in H^{\#}_k(M,\chi \psi^{2})$.
\begin{proof}
Put $m=m_{\psi}$. By using Lemma \ref{lemma:commutation}, we have 
$\Delta_k(f|_kT^{u/m}) = \Delta_k(f)|_kT^{u/m} = 0$ and therefore Lemma \ref{lemma 5.2} will give us
$\Delta_k(f_\psi) = 0$. In view of Remark \ref{lemma 4.1}, by using Lemma \ref{lemma 5.2} the cusp conditions for $f_\psi$ follows. 
Now we only
need to prove the required transformation property for $f_\psi$. 
Let 
\[
\gamma=\begin{pmatrix}
a&b\\cM&d
\end{pmatrix}\in\Gamma_0(M).
\] 
Then
\[
T^{u/m}\gamma T^{-d^2u/m} = \begin{pmatrix}
a'&b'\\c'&d'
\end{pmatrix}\in\Gamma_0(M)\subseteq\Gamma_0(N),
\]
for any $u=1,2 \ldots m$. We have $d'=d-c d^{2} u M / m \equiv d\pmod{m_{\chi}}$. Therefore we get
\[
f|_kT^{u/m}\gamma=\chi(d)f|_kT^{d^2u/m}.
\]
Now by using Lemma \ref{lemma 5.2} and the fact that $\gcd (d,m)=1$, we get 
\[
\left.f_{\psi}\right|_k \gamma=\chi(d) \tau(\overline{\psi})^{-1} \sum_{u=1}^{m} \overline{\psi}(u) \left.f\right|_k T^{d^2 u / m}=\chi(d)\psi(d^{2}) f_{\psi}.
\]
\end{proof}
\end{prop}

\begin{prop}
Let $f$ and $\psi$ be as in Proposition \ref{prop:fpsitrans}. Let $g=\left.f\right|_k \omega_{N}$.
If $\gcd(m_\psi, N)=1$ then
\[
\left.f_{\psi}\right|_k\omega(N m_\psi^{2})=C_{\psi} g_{\bar{\psi}},
\]
where
\begin{equation}
\begin{split}
C_{\psi}=C_{N, \chi, \psi} &=\chi(m_\psi) \psi(-N) \tau(\psi) / \tau(\overline{\psi}) =\chi(m_\psi) \psi(N) \tau(\psi)^{2} / m_\psi.
\end{split}
\label{5.6}
\end{equation}
\label{thm 5.4}
\begin{proof}
Put $m = m_\psi$. Let $u$ be an integer such that gcd$(m,Nu)=1$. Choose $n,v\in\mathbb{Z}$ such that $nm-Nuv=1$. Then from \cite[Eq. 4.3.22]{miyake}, we have 
\[
T^{u / m} \omega\left(N m^{2}\right)=m \cdot \omega(N)\begin{pmatrix}
m & -v \\
-u N & n
\end{pmatrix} T^{v / m}.
\] 
By using Proposition \ref{prop:hashproperty} $(i)$, we have 
\[
f|_kT^{u / m} \omega\left(N m^{2}\right)=\chi(m)g|_kT^{v/m}.
\]
Thus by using Lemma \ref{lemma 5.2} we have
\[
\begin{aligned}
\tau(\overline{\psi}) f_{\psi}|_k \omega(N m^{2}) &=\sum_{u=1}^{m} \overline{\psi}(u) f|_k T^{u / m} \omega\left(N m^{2}\right) \\
&=\chi(m) \sum_{v=1}^{m} \psi(-N v) g|_k T^{v / m} \\&=
\chi(m) \psi(-N) \sum_{v=1}^{m} \psi(v) g|_k T^{v / m} \\
&=\chi(m) \psi(-N) \tau(\psi) g_{\overline{\psi}}.
\end{aligned}
\]
\end{proof}
\end{prop}
\begin{thm}
Let $f$ and $\psi$ be as in Proposition \ref{prop:fpsitrans} and let $g=f|_k \omega(N)$.
Suppose $\gcd(m_\psi, N)=1$. Then each one of the Dirichlet series 
$\Lambda_{N}(f,s,\psi),\Lambda_{N}(g,s,\bar{\psi})$ and $\Omega_{N}(f,s,\psi)$, $\Omega_{N}(g,s,\bar{\psi})$ can be analytically continued to the whole s-plane, is bounded on any vertical strip, and satisfies the following 
functional equation:
\begin{equation*}
\begin{split}
\Lambda_{N}(f,s,\psi)=i^{k} C_{\psi} \Lambda_{N}(g,k-s ,\bar{\psi}),\\
\Omega_{N}(f,s,\psi)=-i^{k} C_{\psi} \Omega_{N}(g,k-s ,\bar{\psi}),
\end{split}
\label{5.7}
\end{equation*}
where $C_{\psi}$ is the constant in \eqref{5.6}. Moreover, we have the direct part of Theorem \ref{thm:main}. 
\label{thm 5.5}
\begin{proof}
By using Proposition \ref{thm 5.4}, we have 
\[
\left.f_{\psi}\right|_k \omega\left(N m_\psi^{2}\right)=C_{\psi} g_{\bar{\psi}}.
\]
The claimed analytic properties now follows from Proposition \ref{prop 5.1}. Combining this with the direct part of Theorem \ref{thm 4.8} we have $(1) \implies (2)$ in Theorem \ref{thm:main}. 
\end{proof}
\end{thm}

\subsection{Proof of Theorem \ref{thm:main}: $(2) \implies (1)$}

Suppose $f$ and $g$ are two formal Fourier series given by \eqref{eq:fgfourier}
satisfying the condition $(2)$ of Theorem \ref{thm:main}. From Lemma \ref{lemma 4.7}, we get that
the Fourier series $f$ and $g$ are absolutely and uniformly convergent on compact subsets of 
$\mathbb{H}$ and define real analytic functions on $\mathbb{H}$. As in the proof of the converse part of 
Theorem \ref{thm 4.8}, by applying $\Delta_k$ on the given 
Fourier series expansions of $f$ and $g$ we get that $\Delta_k(f) = 0 = \Delta_k(g)$.  
By using the converse part of Theorem \ref{thm 4.8} and Proposition \ref{prop 5.1}, we have
\begin{equation*}\label{eq:prelim}
f|_k \omega(N) = g, \ \ f_{\psi}|_k \omega(N m_\psi^{2})=C_{\psi} g_{\bar{\psi}},
\end{equation*} 
for any primitive Dirichlet character $\psi$ with conductor $m_\psi \in \mathcal{P}$. Therefore to establish the converse part of Theorem \ref{thm:main}, we need to verify the required transformation properties with respect 
to $\Gamma_0(N)$-matrices and the cusp conditions. First we prove Proposition \ref{thm 5.11} which will ensure the required cusp conditions by using Lemma \ref{lemma 4.7} if we have the modularity.  

\begin{lemma}
Let $f:\mathbb{H}\rightarrow\mathbb{C}$ be a smooth function given by the Fourier series
\[
f(\tau)=\sum_{n\ll -\infty}c_f^+(n)q^{\frac{n+\kappa}{t}}+c_f^-(0)v^{1-k}q^{\kappa/t}+\sum_{\substack{n\gg \infty\\n\neq 0}}c^-_f(n)\Gamma(1-k,-4\pi nv/t)q^{\frac{n+\kappa}{t}}, 
\]
for some positive real numbers $\kappa, t$, converging uniformly on any compact subset of $\mathbb{H}$. 
Then for any $\tau_0=u_0+iv_0\in\mathbb{H}$, we have 
\[
\frac{1}{t}\int_{\tau_0}^{\tau_0+t}f(\tau)e^{-2\pi i(n+\kappa)\tau/t}d\tau=\begin{cases}
c_f^+(n)+c_f^-(n)\Gamma(1-k,-4\pi nv_0/t)&\text{if $n\neq 0$},\\
c_f^+(0)+c_f^-(0)v_0^{1-k}&\text{if $n=0$}.
\end{cases}
\]
\label{lemma 5.10}
\begin{proof}
Let $m\neq 0$. Then we have 
\[
\begin{split}
\frac{1}{t}\int_{\tau_0}^{\tau_0+t}f(\tau)e^{-2\pi i(m+\kappa)\tau/t}d\tau&=\sum_{n=-\infty}^{\infty}c_f^+(n)
e^{2\pi (m-n)v_0/t}\frac{1}{t}\int_{u_0}^{u_0+t}e^{-2\pi i(m-n)u/t}du\\&+c_f^-(0)v_0^{1-k}e^{2\pi mv_0/t}\frac{1}{t}\int_{u_0}^{u_0+t}e^{-2\pi imu/t}du\\&+\sum_{\substack{n=-\infty\\n\neq 0}}^{\infty}c^-_f(n)\Gamma(1-k,-4\pi nv_0/t)e^{2\pi (m-n)v_0/t}\frac{1}{t}\int_{u_0}^{u_0+t}e^{-2\pi i(m-n)u/t}du\\&=c_f^+(m)+c_f^-(m)\Gamma(1-k,-4\pi mv_0/t).
\end{split}
\]
The case $m=0$ follows similarly.
\end{proof} 
\end{lemma}
\begin{prop}
Let $f:\mathbb{H}\rightarrow\mathbb{C}$ be a smooth function. Assume that $f$ satisfies the first and second conditions of Definition \ref{defn:harmonic}. If $f(\tau)=O(v^{-\sigma})$ as $v\to 0$ for some $\sigma\geq 0$ uniformly in \emph{Re}$(z)$ then $f$ has at most polynomial growth at every cusp of $\Gamma_0(N)$.
\label{thm 5.11}
\begin{proof}
Since the cusp $i\infty$ is $\Gamma_0(N)$-equivalent to a real cusp, it is enough to establish the lemma for real cusps. Let 
$\rho\in \mathbb{Q}$ be a cusp of $\Gamma_0(N)$ with width $t$ and parameter $\kappa$.
Let $\alpha = \left(\begin{smallmatrix}
a&b\\c&d\end{smallmatrix}\right) \in \mathrm{SL}_2(\mathbb{Z}), \ c \neq 0$, such that 
$\rho=\alpha(i\infty)$. 
Then by following the proof of Lemma \ref{fe}, we have  
\[
(f|_k\alpha)(\tau)=\sum_{n=-\infty}^{\infty}c_f^+(n)q^{\frac{n+\kappa}{t}}+c_f^-(0)v^{1-k}q^{\kappa/t}+\sum_{\substack{n=-\infty
\\n\neq 0}}^{\infty}c^-_f(n)\Gamma(1-k,-4\pi nv/t)q^{\frac{n+\kappa}{t}}
\]
which converges uniformly on any compact subset of $\mathbb{H}$. Note that $\kappa \in [0,1)$. Now by using Lemma \ref{lemma 5.10}, for any 
$\tau_0=u_0+i v_0 \in \mathbb{H}$ we have 
\begin{equation}
\frac{1}{t}\int_{\tau_0}^{\tau_0+t}(f|_k\alpha)(\tau)e^{-2\pi i(n+\kappa)\tau/t}d\tau=\begin{cases}
c_f^+(n)+c_f^-(n)\Gamma(1-k,-4\pi nv_0/t)&\text{if $n\neq 0$},\\
c_f^+(0)+c_f^-(0)v_0^{1-k}&\text{if $n=0$}.
\end{cases}
\label{5.8}
\end{equation}
Since $c \neq 0$, We have 
\[
\text{Im}(\alpha\tau)=\frac{v}{|c\tau+d|^2} \rightarrow 0 ~~\text{as}~ v\rightarrow \infty
\]    
uniformly on $|{\rm Re}(\tau)| \leq t/2$.
Therefore we have 
\[
(f|_k\alpha)(\tau)=(c\tau+d)^{-k}f(\alpha\tau)=O(|c\tau+d|^{-k} \ {\rm Im}(\alpha\tau)^{-\sigma})=O(v^{\sigma-k})~~\text{as}~v\rightarrow\infty
\]
uniformly on $|\text{Re}(\tau)|\leq t/2$. By taking $\tau_0=iy-t/2$, we see that the left hand side of \eqref{5.8} is
$O(y^{\sigma-k}e^{2\pi (n+\kappa)y/t}) \ \ {\rm as}\ y \rightarrow \infty$. Therefore for any $n \in \mathbb{Z}, n \neq 0$,\eqref{5.8} gives us 
\begin{equation}\label{eq:last11}
\lim_{y\to\infty}\left|\frac{c_f^+(n)+c_f^-(n)\Gamma(1-k,-4\pi ny/t)}{y^{\sigma-k}e^{2\pi (n+\kappa)y/t}}\right|<\infty.
\end{equation}
By using the asymptotic relation \eqref{gamma} for $\Gamma(1-k,-4\pi n y/t)$, we have
\[
\Gamma(1-k,-4\pi n y/t) \sim (-4\pi n y /t)^{-k} e^{4 \pi n y/t} \ \ {\rm as}\ y \rightarrow \infty.
\]
As $\kappa \in [0,1)$, if $n<0$ then the second term in \eqref{eq:last11} goes to $0$ but the first term goes to $\infty$ 
as $y \rightarrow \infty$ unless $c_f^+(n) = 0$. We argue similarly in the case of $n > 0$ to get $c_f^-(n) = 0$.  
\end{proof}
\end{prop}

We now proceed to establish the modularity of $f$. 
Let us begin by introducing some terminology. For two integers
$m, v$ with gcd$(m,Nv)=1$, take $n,u\in\mathbb{Z}$ such that $mn-Nuv=1$. Put 
\[
\gamma(m,v)=\begin{pmatrix}
m&-v\\-Nu&n
\end{pmatrix}\in\Gamma_0(N).
\] 
Although the choice of $\gamma(m,v)$ is not unique but for any such matrix, we have
\[
T^{u / m} \omega\left(N m^{2}\right)=m \cdot \omega(N)\gamma(m,v)T^{v / m},
\]
as in the proof of Proposition \ref{thm 5.4}. Moreover, $u\bmod m$ is uniquely determined. 

\begin{lemma}\label{lemma:5.6}
Let $m$ be an odd prime number or 4 coprime to $N $. Suppose that $f$ and $g$ satisfy $f_{\psi}|_k \omega\left(N m^{2}\right)=C_{\psi} g_{\bar{\psi}}$ for all primitive Dirichlet characters $\psi$ mod $m$ with the constant
\[
C_{\psi}=\chi(m) \psi(-N) \tau(\psi) / \tau(\overline{\psi}),
\]
then
\[
\left.g\right|_k(\chi(m)-\gamma(m, u)) T^{u / m}=\left.g\right|_k(\chi(m)-\gamma(m, v)) T^{v / m}
\]
for any two integers $u$ and $v$ coprime to $m$.
\begin{proof}
The proof is about handling the slash operators and therefore it will be exactly on the lines of 
\cite[Proof of Lemma 4.3.13]{miyake}.
\end{proof}
\end{lemma}
\begin{lemma}
Let $m$ and $n$ be odd prime numbers or 4. Assume that both $m$ and $n$ are coprime to $N$. 
Then the matrix 
\[
\beta(m,n,u,v):=\begin{pmatrix}
1 & -2 v / m \\
2 u N / n & (4 / m n)-3
\end{pmatrix}
\] 
is an elliptic matrix whose eigenvalues are not roots of unity. Moreover,
if $f$ and $g$ satisfy  $f_{\psi}|_k \omega\left(N m^{2}\right)=C_{\psi} g_{\bar{\psi}}$ for all primitive Dirichlet characters $\psi$ whose conductor $m_{\psi}=m$ or $n,$ with the constant
\[
C_{\psi}=\chi(m) \psi(-N) \tau(\psi) / \tau(\overline{\psi}),
\] then
\[
\left(g\bigg{|}_k\begin{pmatrix}
m&-v\\-Nu&n
\end{pmatrix}-\chi(m)g\right)|_k(1-\beta(m,n,u,v))=0.
\]
\label{prop 5.7}
\begin{proof}
We get this lemma by following the first half of the proof of \cite[Lemma 4.3.14]{miyake} and 
by using Lemma \ref{lemma:5.6} appropriately.
\end{proof}
\end{lemma}

Here we record the following theorem of Neururer and Oliver (cf. \cite[Theorem 3.11]{NO}) which will be useful in proving our converse theorem.
\begin{thm}
If $h$ is a continuous function on $\mathbb{H}$ that is invariant under two infinite order elliptic matrices with distinct fixed points in $\mathbb{H}$ then it is a constant.
\label{thm 5.8}
\end{thm} 

Now we prove the following theorem which is the main step in the proof of the converse part. 
\begin{thm}
Let $m, n, f, g, c_\psi$ be same as in Lemma \ref{prop 5.7}. Then we have
\[
\left.g\right|_k \gamma=\overline{\chi}(n) g
\]
for all $\gamma \in \Gamma_{0}(N)$ of the form $\gamma=\begin{pmatrix}
m&-v\\-Nu&n
\end{pmatrix}.$
\label{thm 5.9}
\begin{proof}
Put 
\[
h=g\bigg{|}_k\begin{pmatrix}
m & -v \\
-Nu & n
\end{pmatrix}-\chi(m)g.
\]
Then by using Lemma \ref{prop 5.7} we get that $h$ is invariant under the elliptic matrix 
\[
\beta(m,n,u,v)=\begin{pmatrix}
1 & -2 v / m \\
2 u N / n & (4 / m n)-3
\end{pmatrix}.
\]
Since $\beta(m, n, u,v)$ has eigenvalues which are not roots of unity, it is an infinite order matrix. The fixed point of $\beta(m,n, u, v)$ in $\mathbb{H}$ is given by
\[
z_{1}=i \sqrt{-\frac{n^{2}\left(\frac{1}{mn}-1\right)^{2}}{N^{2} u^{2}}+\frac{n v}{m N u}}-\frac{n\left(\frac{1}{mn}-1\right)}{Nu}.
\]
Let $m'=m-rNu$ be an odd prime for some non-zero integer $r$ such that $m'$ is coprime to $N$ and distinct from 
$m$ and $n$. Put $v'=v - r n$. Then we have 
\[
g\bigg{|}_k\left(\begin{array}{cc}
m' & -v' \\
-Nu & n
\end{array}\right)=g\bigg{|}_k\left(\begin{array}{cc}
1 & r \\
0 & 1
\end{array}\right)\left(\begin{array}{cc}
m & -v \\
-Nu & n
\end{array}\right)=g\bigg{|}_k\left(\begin{array}{cc}
m &- v \\
-Nu & n
\end{array}\right).
\]
Since $\chi(m')=\chi(m)$, by using Lemma \ref{prop 5.7} we see that
\[
h'=g\bigg{|}_k\begin{pmatrix}
m' & -v' \\
-Nu & n
\end{pmatrix}-\chi(m')g=h
\]
is invariant under the elliptic matrix 
\[
\beta(m',n, u,v')=\begin{pmatrix}
1 & -2 v' / m' \\
2 u N / n & (4 /m'n)-3
\end{pmatrix}.
\]
The fixed point of $\beta(m', n, u,v')$ in $\mathbb{H}$ is given by
\[
z_{2}=i \sqrt{-\frac{n^{2}\left(\frac{1}{m'n}-1\right)^{2}}{N^{2} u^{2}}+\frac{nv'}{m' N u}}-\frac{n\left(\frac{1}{m'n}-1\right)}{Nu}.
\]
By comparing the real parts we see that $z_1\neq z_2.$ Thus by Theorem \ref{thm 5.8} we get that $h$ is a constant. 
Put $h(\tau)=c$, a constant in $\mathbb{C}$. Let $\tau_0$ be the unique fixed point of $\beta(m, n, u,v)$ in 
$\mathbb{H}$ and put
\[
\rho:=\frac{1}{\tau_0-\bar{\tau}_0}\begin{pmatrix}
1&-\tau_0\\1&-\bar{\tau}_0
\end{pmatrix}\in \mathrm{GL}(2,\mathbb{C}),
\]
\[
p(w):=\left(\left.h\right|_k \rho^{-1}\right)(w):=(1-w)^{-k}h\left(\rho^{-1} w\right)=(1-w)^{-k}c,~~~w\in\{z\in\mathbb{C}:|z|<1\}.
\]
We write 
\[
\rho\beta(m, n, u,v)\rho^{-1}=\begin{pmatrix}
\zeta&0\\0&\zeta^{-1}
\end{pmatrix},
\]
where $\zeta$ is an eigenvalue of $\beta(m, n, u,v)$. Since $h$ is invariant under $\beta(m, n, u,v)$, we have 
\[
p(\zeta^2w)=\zeta^{-k}p(w).
\] 
In particular, for $w=0$ we have 
$c=\zeta^{-k}c$ which implies that $c=0$ as
$\zeta$ is not a root of unity. We conclude the proof by observing that $\chi(m)={\bar{\chi}}(n)$.
\end{proof}
\end{thm}

Finally, we deduce the modularity of $f$  from the modularity of $g$ as in the proof of 
Lemma \ref{lemma 4.2}. 
We get the modularity of $g$ by proceeding exactly similar to the proof of the modular property in \cite[Theorem 4.3.15]{miyake} and by making use of Theorem \ref{thm 5.9} appropriately.

\bigskip

\noindent{\bf {Acknowledgements.}}
We would like to thank the anonymous referees for detailed comments and suggestions which led to some mathematical corrections and an improvement of the presentation.

\bigskip 

\noindent{{\bf Data availability.}} 
Data sharing not applicable to this article as no datasets were generated or analysed during the current study.

\bibliographystyle{plain}

\bibliography{mybib}

\end{document}